\pdfoutput=1
\RequirePackage{snapshot}
\documentclass[10pt]{iopart}

\usepackage[utf8]{inputenc}
\usepackage[T1]{fontenc}

\usepackage[pdftex]{hyperref}
\usepackage{setstack}
\usepackage{tikz}
\usepackage{pgfplots}

\newlength{\myx} % Variable zum Speichern der Bildbreite
\newlength{\myy} % Variable zum Speichern der Bildhöhe
\newcommand\includegraphicstotab[2][\relax]{%
\settowidth{\myx}{\includegraphics[{#1}]{#2}}%
\settoheight{\myy}{\includegraphics[{#1}]{#2}}%
\parbox[c][1.1\myy][c]{\myx}{%
\includegraphics[{#1}]{#2}}%
}% Ende neuer Befehl

\usepgfplotslibrary{external} 
\expandafter\let\csname equation*\endcsname\relax 
\expandafter\let\csname endequation*\endcsname\relax 
\expandafter\let\csname leftroot\endcsname\relax 
\expandafter\let\csname uproot\endcsname\relax 
\expandafter\let\csname boxed\endcsname\relax
\expandafter\let\csname dddot\endcsname\relax
\expandafter\let\csname ddddot\endcsname\relax
\expandafter\let\csname overset\endcsname\relax 
\expandafter\let\csname underset\endcsname\relax 
\expandafter\let\csname sideset\endcsname\relax 
\expandafter\let\csname subarray\endcsname\relax 
\expandafter\let\csname endsubarray\endcsname\relax
\expandafter\let\csname substack\endcsname\relax
\usepackage{amsmath,amssymb,epsfig}

\usepackage{xcolor}

\usepackage{amsthm}

\newtheorem{theorem}{Theorem}[section]

\newtheorem{lemma}[theorem]{Lemma}

\theoremstyle{definition}

\mathchardef\ordinarycolon\mathcode`\:
\mathcode`\:=\string"8000\def\R{\mathbb{R}}
\begingroup \catcode`\:=\active\def\Z{\mathbb{Z}}
  \gdef:{\mathrel{\mathop\ordinarycolon}}\def\N{\mathbb{N}}
\endgroup

\newcommand{\N}{\mathbb{N}}
\newcommand{\Z}{\mathbb{Z}}

\renewcommand{\R}{\mathbb{R}}

\usepackage{mathrsfs}
\newcommand{\Fc}{\mathcal{F}}

\newcommand{\Mc}{\mathcal{M}}
\newcommand{\Nc}{\mathcal{N}}
\newcommand{\Oc}{\mathcal{O}}

\newcommand{\Rc}{\mathcal{R}}

\newcommand{\bdot}{\boldsymbol{\cdot}}                % Mittelgr. zentralerPunkt
\newcommand{\norm}[1]{\left\lVert#1\right\rVert}      % Norm
\newcommand{\abs}[1]{\left|#1\right|}                 % Absolutbetrag
\newcommand{\paren}[1]{\left(#1\right)}               % Klammern
\newcommand{\h}[1]{\left\langle#1\right\rangle}       % Skalarprodukt
\renewcommand{\d}{\,\mathrm{d}}						  % within the integral sign
\newcommand{\conv}{\ast}
\DeclareMathOperator*{\argmin}{\mathrm{arg\:min}}
\renewcommand{\triangle}{\vartriangle}
\renewcommand{\epsilon}{\varepsilon}
\renewcommand{\rho}{\varrho}
\renewcommand{\bar}[1]{\overline{#1}}
\renewcommand{\tilde}{\widetilde}
\renewcommand{\th}{\theta}

\newcommand{\be}{\begin{equation}}
\newcommand{\ee}{\end{equation}}
\newcommand{\bea}{\begin{eqnarray}}
\newcommand{\eea}{\end{eqnarray}}
\newcommand{\bean}{\begin{eqnarray*}}
\newcommand{\eean}{\end{eqnarray*}}

\newcommand{\bel}[1]{\begin{equation}\label{#1}}

\newcommand{\eel}[1]{{\label{#1}\end{equation}}}

\newcommand{\prox}{\operatorname{prox}}
\newcommand{\p}[1]{\left(#1\right)}

\usepackage[babel]{csquotes}
\usepackage[labelfont=bf,font=small,labelsep=space]{caption}
\usepackage[font=footnotesize]{subcaption}
\usepackage{array,booktabs, tabularx}
\usepackage{longtable}

\usepackage{pgfplots}

\newcommand{\RI}{\operatorname{RI}}
\DeclareMathOperator*{\minimize}{minimize}

\begin{document}

\title{Joint Image Reconstruction and Segmentation Using the Potts Model}

\author{Martin Storath$^1$,  Andreas Weinmann$^{2,4}$, J\"urgen Frikel$^{3,4}$,  Michael Unser$^1$}
\address{
$^1$Biomedical Imaging Group, École Polytechnique Fédérale de Lausanne, Switzerland \\
 $^2$Department of Mathematics, Technische Universit\"at M\"unchen, Germany \\
 $^3$Department of Mathematics, Tufts University, USA \\
 $^4$Research Group Fast Algorithms for Biomedical Imaging, Helmholtz Zentrum M\"unchen, Germany
}
\ead{martin.storath@epfl.ch, andreas.weinmann@tum.de, juergen.frikel@helmholtz-muenchen.de,  michael.unser@epfl.ch}

\begin{abstract}
We propose a new algorithmic approach to the non-smooth and non-convex
Potts problem (also called piecewise-constant Mumford-Shah problem) 
for inverse imaging problems.
We derive a suitable splitting into specific subproblems that can all be solved efficiently.
Our method does not require a priori knowledge  
on the gray levels nor on the number of segments of the reconstruction.
Further, it avoids anisotropic artifacts such as geometric staircasing.
We demonstrate the suitability of our method for joint image reconstruction and segmentation.
We focus on Radon data,
where we in particular consider limited data situations.
For instance, our method is able to recover all segments of the Shepp-Logan phantom 
from $7$ angular views only.
We illustrate the practical applicability on a real PET dataset.
As further applications, we consider spherical Radon data 
as well as blurred data.
\end{abstract}

\vspace{2ex} \noindent{\it Keywords}: Potts model, piecewise-constant Mumford-Shah model, regularization of ill-posed problems, image segmentation,  Radon transform, spherical Radon transform, computed tomography, photoacoustic tomography, deconvolution.

\maketitle

%===================================================================================================
\section{Introduction}

In this paper, we consider ill-posed imaging problems with incomplete data.
Incomplete data are often due to technical restrictions or design issues of the imaging modality, like in
the case of freehand SPECT \cite{Wendler:2010cb} or limited-angle tomography \cite{Buehler:2011bv, Natterer86,Wang:2009uz}.
Also health-related considerations lead to incomplete data. 
For example, sparse-angle setups are used to reduce radiation doses in  x-ray tomography \cite{barkan2014}.
In addition to incompleteness, the data is usually corrupted by noise and may also suffer from blur \cite{ Hahn:il,  Kohr:2011cz, Rosenthal:2011ha, Wolf:2013cy}. 
Altogether, this makes the reconstruction problems severely ill-posed
which means that small perturbations in the data potentially lead to large errors in the reconstruction \cite{Davison83,Louis:1984vy, Louis_SVD_of_limited_angle_transform86, Natterer86}.
It may even happen that certain features (singularities) are invisible from the incomplete data \cite{ Frikel:2013gb,Katsevich:1997js, Quinto93}.
Consequently, 
the quality of reconstruction  decreases significantly 
and one tends to lose fine details.
Nevertheless, one can still try to reconstruct 
the object at a coarser scale.
 This is often of particular interest in medical imaging;
for example, the locations of inner organs might be needed for surgery planning \cite{ramlau2007mumford}.

Classical reconstruction methods  perform poorly
in limited-data situations.
Better approaches 
incorporate specific prior assumptions on the reconstruction.
They are typically stated in terms of the minimization 
of some cost function. A popular representative is the convex total variation \cite{chambolle2011first, rudin1992nonlinear}. 
 A recent trend is to use non-convex regularizers \cite{bostan2013sparse, chartrand2007exact,nikolova2008efficient,ramlau2007mumford}. 
Although analytically and computationally more demanding, they
give more freedom in the modeling and often yield better reconstructions
\cite{chartrand2007exact,chartrand2009fast}.
In their seminal work \cite{mumford1989optimal}, 
Mumford and Shah introduced a cost functional on the piecewise-constant functions -- now called the piecewise-constant Mumford-Shah functional -- where the length of the discontinuity set is penalized.
This functional has shown good performance,
especially for the recovery of geometric macrostructures from poor data
\cite{klann2011mumford,ramlau2007mumford,ramlau2010regularization}.
The piecewise-constant Mumford-Shah model also appears in statistics and image processing 
where it is often called \emph{Potts model} \cite{boykov2001fast, boysen2009jump, boysen2009consistencies,pock2009convex,winkler2003image}.
The variational formulation of the Potts model is given by 
\begin{equation}\label{eq:pottsRadon}
	\textstyle\argmin_u \ \gamma \, \| \nabla u\|_{0} + \norm{A u - f}_{2}^2.
\end{equation}
Here, $A$ is a linear operator (e.g.,~the Radon transform) and $f$ is an element of the data space (e.g.,~a sinogram).
A mathematically precise definition of the jump term $\| \nabla u \|_{0}$ is rather technical in a spatially continuous setting.
However, if $u$ is piecewise-constant and  the discontinuity set of $u$ is sufficiently regular, say, a union of $C^1$ curves, then $\| \nabla u\|_{0}$ is just the total arc length of this union.
In general, the gradient $\nabla u$ is given in the distributional sense and the boundary length is expressed in terms of the \mbox{$(d-1)$-dimensional} Hausdorff measure.
When $u$ is not piecewise-constant, the jump penalty is infinite \cite{ramlau2010regularization}.
The second term measures the  fidelity of a solution $u$ to the data $f.$ 
The parameter $\gamma > 0$ controls the balance between data fidelity and jump penalty.
The Potts model can be interpreted in two ways.
On the one hand, if the imaged object is (approximately) piecewise-constant, 
then the solution is an (approximate) reconstruction of the imaged object.
On the other hand, since  a piecewise-constant solution 
directly induces a partitioning of the image domain,
it can be seen as joint reconstruction and segmentation.
Executing reconstruction and segmentation jointly typically leads to better results
than performing the two steps successively \cite{klann2011mumford,ramlau2007mumford};
see Figure \ref{fig:intro}.

    \begin{figure}[t]
\def\figwidth{0.23\columnwidth}
\def\figspace{\hspace{0.1\textwidth}}
\centering
	\def\figfolder{newExperiments/segmentationOfFBP/}
	\begin{subfigure}[t]{\figwidth}
		\centering
	\includegraphics[width=1\columnwidth]{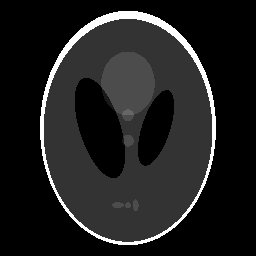}
	\caption{Original ($256 \times 256$)}
	\end{subfigure}
	\hfill
	\begin{subfigure}[t]{\figwidth}
	\includegraphics[width=\columnwidth]{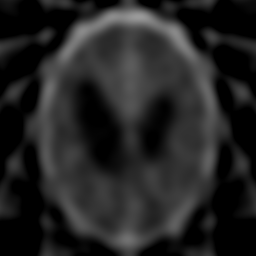}
	\caption{FBP reconstruction from $7$ angles (Hamming window, tuned w.r.t. PSNR).}
	\end{subfigure}
	\hfill
		\begin{subfigure}[t]{\figwidth}
	\includegraphics[width=\columnwidth]{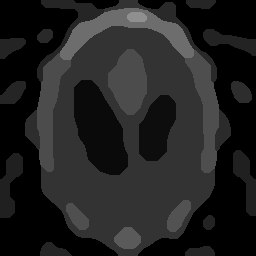}
	\caption{Graph-cut based segmentation of FBP result \\
(Rand index: $\protect\input{\figfolder randIndexGC.txt}).$ }
	\end{subfigure}
	\hfill
		\def\figfolder{newExperiments/radonAnglesNew/nAngles7A2/}
		\begin{subfigure}[t]{\figwidth}
	\includegraphics[width=\columnwidth]{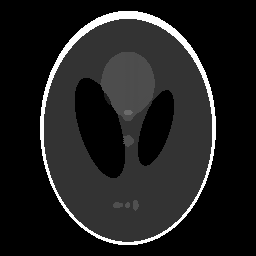}
	\caption{Joint reconstruction and segmentation using our method (Rand index: $\protect\input{\figfolder randIndexPotts.txt}).$}
	\end{subfigure}
	\vspace{-0.2cm}
	\caption{
Segmentation from highly undersampled Radon data 
($\protect\input{\figfolder nAngles.txt}$ projection angles). 
Reconstruction by a classical method and 
subsequent segmentation leads to unsatisfactory results.
Our method produces a high-quality segmentation.
}
	\label{fig:intro}
\end{figure}

The Potts problem is algorithmically challenging. 
For $A = \mathrm{id},$ it is NP-hard  in dimension greater than one \cite{veksler1999efficient}, 
and, for general linear operators $A,$ it is even NP-hard for one-dimensional signals \cite{storath2013jump}. Thus, there is no hope to find a global minimizer in reasonable time.
Nevertheless, due to its importance in image reconstruction and segmentation,
several approximative strategies have been proposed.
Bar et al.~\cite{bar2004variational} consider an Ambrosio-Tortorelli-type approximation. Kim et al.~\cite{kim2002curve} use a level-set based active contour method 
 for deconvolution.
Ramlau and Ring \cite{ramlau2007mumford} employ a related level-set approach 
for the joint reconstruction and segmentation of x-ray tomographic images;
further applications are electron tomography \cite{klann2011mumfordElectron} and SPECT \cite{klann2011mumford}. 
The latter authors were the first to investigate the regularizing properties of such functionals \cite{ramlau2007mumford, ramlau2010regularization}.
We elaborate further on Potts regularization in inverse problems and on 
existing algorithmic approaches in the sections 1.1 and 1.2, respectively.

In this paper, we first discretize the Potts problem as
\begin{equation}\label{eq:pottsWeight}
  u^* = \argmin_{u \in \R^{m\times n}} \gamma  \sum_{s = 1}^S \omega_s \|  \nabla_{p_s} u\|_0 + \| Au - f\|_2^2.
\end{equation}
The symbol $\nabla_{p}$ denotes finite differences with respect to the displacement vector $p$
so that $\nabla_{p} u = u( \cdot + p) - u,$ where $p \in \Z^2\setminus \{0\}.$
The symbol $\| \nabla_{p_s} u \|_0 $ denotes the number of nonzero entries of  $\nabla_{p_s} u.$
The displacement vectors $p$ belong to a neighborhood system $\Nc = \{ p_1, ...,p_S\}$ and 
$\omega_1,..., \omega_S$ are nonnegative weights.
The simplest neighborhood system is made up of the two canonical basis vectors of $\R^2$ along with unit weights.
Unfortunately, when refining the grid, this discretization converges to a limit that measures the boundary in terms
of the $\ell^1$ analogue of the Hausdorff measure \cite{chambolle1995image}.
The practical consequences  are unwanted block artifacts in the reconstruction (geometric staircasing).
The addition of diagonal or \enquote{knight-move} finite differences (referring to the moves of a knight in chess) mildens such anisotropy effects \cite{chambolle1999finite}.
We provide a general scheme for the proper choice of finite-difference systems and accompanying weights 
which allows for arbitrarily good approximations of the Euclidean length.

Based on \eqref{eq:pottsWeight}, we 
propose a new minimization strategy for the Potts problem.
Our key contribution is a particularly  suitable splitting of the Potts problem \eqref{eq:pottsWeight} 
into specific subproblems. 
The first subproblem is a classical Tikhonov-regularized problem
with a solution that reduces to a linear system of equations.
All the remaining subproblems can be solved efficiently by
dynamic programming \cite{chambolle1995image,friedrich2008complexity, mumford1989optimal, storath2014fast}. 
We prove that our algorithm converges.
A major advantage of our method is that 
neither the number of segments 
nor the gray-values of a solution have to be fixed a priori.
Further, the method does not require any initial guess of the solution.
Last but not least, it is highly parallelizable and easy to implement.

We demonstrate the suitability of our method for 
joint image reconstruction and segmentation in several  setups.
We consider  Radon data, 
which appear in x-ray tomography (CT) and in positron emission tomography (PET).
In noise-free conditions, we achieve an almost perfect reconstruction/segmentation 
of the Shepp-Logan phantom from as few as seven projections 
(see Figure \ref{fig:intro}).
Also in the presence of noise,
our method yields a high-quality segmentation
from a small number of projections.
On PET data of a physical phantom, we obtain a reliable segmentation of the anatomic structures.
As further applications, we briefly discuss photoacoustic tomography, which is based on spherical Radon data,
and the deconvolution of vector-valued data.

\subsection{Potts Regularization in Inverse Problems}
\label{sec:pottsRegInInvPr}

In many imaging problems, the imaging operator $A$ is not boundedly invertible.
Examples are the Radon transform and the spherical Radon transform when viewed as operators on the corresponding $L^2$ spaces.
Because of  unboundedness, a direct inversion (if possible at all) may amplify small perturbations in the data.
In order to attenuate such effects and produce more stable reconstructions, regularization is needed. 
A popular approach for regularization
is by means of minimizing an energy functional of the form
$\gamma J(u) + \|Au - f\|_2^2.$
Here, the second term is the data-fidelity term 
while the first term -- called the regularizing term -- is a penalty that incorporates a priori knowledge on the solution. Classical regularizations are Besov or Sobolev seminorms which associate certain smoothness classes with 
the solution. A less classical choice is the $TV$ seminorm which leads to solutions of bounded variation.
In the context of sparsity regularization,  non-convex functionals are also used as regularizing terms \cite{bredies2009regularization}. 
The Potts functional \eqref{eq:pottsRadon} fits into this framework by letting $J$ be the jump penalty. 
It is non-convex and may be seen as a term that enforces a sparse gradient. 
We note that Mumford-Shah approaches (which include Potts functionals) 
also regularize the boundaries of the discontinuity set of the underlying signal \cite{jiang2014regularizing}.

The existence of minimizers of \eqref{eq:pottsRadon} is not guaranteed in a continuous domain setting
\cite{fornasier2010iterative,fornasier2013existence, ramlau2010regularization, storath2013jump}.
For example, if $A$ is a compact convolution operator originating from a smooth function and if data is given by the impulse response $f = A \delta,$ then the associated Potts functional does not have a minimizer \cite{storath2013jump}.
In order to ensure the existence of minimizers,
additional penalty terms such as an $L^p$ ($1<p<\infty$) term of the form 
$\|u\|_p^p$ \cite{ramlau2007mumford, ramlau2010regularization}
or pointwise boundedness constraints \cite{jiang2014regularizing} have been considered.
We note that the existence of minimizers is guaranteed in the discrete domain setup \cite{fornasier2010iterative, storath2013jump}.

It is important to verify that
the Potts model is a regularization method in the sense of inverse problems.
The first work dealing with this task is \cite{ramlau2010regularization}. 
Ramlau and Ring \cite{ramlau2010regularization} assume that the solution space consists of non-degenerate piecewise-constant functions with at most $k$ (arbitrary, but fixed) different values 
which are additionally bounded. Under relatively mild assumptions on the operator $A$, they show stability. In addition, they   
give a parameter choice rule and show that the solutions corresponding to the chosen parameters (which depend on the level of noise) converge to a noise-free solution as the noise vanishes.  
This means that the method is a regularizer in the sense of inverse problems.
Related references are \cite{klann2013regularization, klann2011mumford} and the recent publication \cite{jiang2014regularizing}
which includes (non-piecewise-constant) Mumford-Shah functionals.

\subsection{Existing Algorithmic Approaches to Potts and Related Problems}
\label{sec:ExistingAlgs}

The name \enquote{Potts model} for \eqref{eq:pottsRadon} has been retained in honor of R.B.~Potts \cite{potts1952some} who, as early as 1952, considered related jump penalties in his work in the field of statistical mechanics.
The classical Potts model ($A = \text{id}$) 
was first introduced in image processing by Geman and Geman \cite{geman1984stochastic} in a statistical framework. Their algorithmic approach is based on simulated annealing. 
From a variational-calculus point of view, the problem was first considered by Mumford and Shah \cite{mumford1989optimal}. Ambrosio and Tortorelli  \cite{ambrosio1990approximation} proposed an approximation by elliptic functionals.
Currently, popular algorithmic approaches for the classical case $A = \text{id}$ are based on active contours \cite{chan2001active, cremers2007review}, graph cuts \cite{boykov2001fast},
convex relaxations \cite{pock2009convex, chambolle2012convex},
and semi-global matching \cite{hirschmuller2008stereo}. 

The general case of $A$ being a linear operator has been investigated to a lesser extent.
Yet,  strategies based on active contours emerge as an important class.
The idea is to parameterize the jump set 
of $u$ by a set of contours  which evolve according to a deformation force. Active contours are used for $A$ being a convolution operator in \cite{bar2004variational}  and for $A$ being the Radon transform in \cite{ramlau2007mumford}. Both use level sets to parameterize the active contours.
A similar method has been applied to joint segmentation of SPECT/CT data \cite{klann2013regularization, klann2011mumford}.
Level-set  methods have also been  applied to stabilize sparse-angle tomography \cite{kolehmainen2008limited}.
In \cite{kreutzmann2014geometric}, the authors use explicitly parameterized contours for the application to bioluminescence tomography.
In general, active contours are quite flexible as the deformation force can be easily adjusted.
Their main disadvantages are that they require a good guess on the initial contour and a good guess on the expected number of gray values of the solution.

Graph cuts \cite{boykov2001fast} are a popular  strategy to address the classical Potts problem with $A=\text{id}$.
In \cite{storath2013jump}, the inverse problem for  a general $A$
is approached by iteratively
using a graph-cut strategy on a classical Potts problem, followed by Tikhonov regularization. 
There, the authors apply their algorithm to deconvolution.
A drawback of this approach is that graph cuts need an (a priori) discretization of the codomain of $u.$ 
Thus, one needs either a good initial guess on the values   that $u$ may take, or a very fine (and expensive) discretization of the codomain of $u.$

To circumvent NP hardness \cite{boykov2001fast}, 
the jump penalty is often replaced by the total variation $\| \nabla u \|_1;$ see \cite{ beck2009fast, bredies2013regularization, bronstein2002reconstruction,candes2006robust, defrise2011algorithm,yan2011expectation} and the references therein.
TV minimization has been used for the reconstruction from sparse Radon data in \cite{sidky2006accurate,sidky2008image, sidky2011constrained}.
TV minimization is  theoretically and algorithmically easier to access since it is a convex problem. 
The theory of compressed sensing gives conditions when the minimizers of the relaxed problem coincide with those of the original problem  \cite{candes2006robust, needell2013stable}.
However, the conditions are relatively restrictive and fail to apply to some problems of practical importance.
In limited-data situations, TV minimizers typically lack sharply localized boundaries \cite{chartrand2007exact, chartrand2009fast}. 
To sharpen the results of total variation minimization, various techniques such as iterative reweighting \cite{candes2008enhancing}, 
simplex constraints \cite{lellmann2009convex}, or iterative thresholding \cite{cai2013multiclass}  have been proposed. 

In order to come closer to the solution of the initial Potts problem,
many authors propose the use of non-convex priors.
Chartrand \cite{chartrand2007exact} uses priors based on the $\ell^p$ norm of the gradient, $0 < p < 1,$
for the reconstruction of MRI-type images. The x-ray CT setup was considered in \cite{chartrand2013nonconvexXRAY, SidkyChJoPan:2013ti}. 
Logarithmic priors are considered by Bostan et al.~\cite{bostan2013sparse}
for the reconstruction of biomedical images. 
Nikolova et al.~\cite{nikolova2010fast, nikolova2008efficient} 
propose a whole class of  non-convex regularizers which
are treated using a graduated non-convexity approach.

Another  approach is to 
transform the Potts problem to an $\ell^0$ problem 
 \cite{artina2013linearly,chartrand2009fast, fornasier2010iterative}. 
The resulting problem is separable
which allows for the application of iterative hard-thresholding-type algorithms 
\cite{artina2013linearly, blumensath2008iterative, blumensath2009iterative, fornasier2010iterative}. 
However, using this transformation comes with constraints in form of discrete Schwarz conditions \cite{fornasier2010iterative} as well as a data term of the form $\|ABu-f\|_2^2$ 
with  a full triangular matrix $B.$ 
While the initial system matrix $A$ is typically sparse, the modified matrix $AB$ is not so in general. 

\subsection{Organization of the Article}

In Section 2, we present our splitting approach to the Potts problem.
We start by explaining the basic approach using an anisotropic discretization of \eqref{eq:pottsRadon}.
Next, we discuss strategies to get more isotropic discretizations, thus attenuating
the unwanted geometric staircasing effect resulting from an anisotropic discretization. 
Then, we present our general algorithm.
We briefly discuss more general data terms
and prove the convergence of our algorithm.
In Section 3, we apply our method to ill-posed imaging problems.
In particular, we consider Radon data as well as spherical Radon data. Furthermore,
we apply our technique to real PET data. Eventually, we apply our method to deconvolution problems.

\section{A Splitting Approach for the Potts Problem}

In this section, we present our splitting approach for the discrete-domain 
Potts functional \eqref{eq:pottsWeight}.
It seems instructive to first describe the basic idea of the splitting 
in the simplest case, which is the anisotropic discretization of the length term.
This discretization turns out to be anisotropic 
as it measures the length of the discrete boundary in the Manhattan metric, which is the metric induced by the $\ell^1$ norm on $\R^2$.
This typically leads to block artifacts in the reconstruction.
To avoid this we derive appropriate  neighborhood systems and corresponding weights
such that the discrete length term in \eqref{eq:pottsWeight} becomes more isotropic.
Based on this discretization, we  formulate our general  splitting for the Potts problem. 
The problem reduces to smaller tractable subproblems that we briefly describe. 
We conclude the section with a convergence result.

\subsection{Basic Splitting Algorithm for an Anisotropic Discretization}\label{sec:basicSplitting}

In the simplest case, the discretization of the regularizing term $\| \nabla u\|_0$
uses only  finite differences with respect to the coordinate axes.
Thus, $p_1 = (1, 0),$ $p_2 = (0,1),$ and the weights $\omega_1, \omega_2$ are equal to $1.$ 
Then, the regularizing term reads 
\[
	\| \nabla u \|_0 = \| \nabla_{p_1} u \|_0 + \| \nabla_{p_2} u \|_0 =
	| \{ (i,j) : u_{ij} \neq u_{i+1,j} \} | + | \{ (i,j) : u_{ij} \neq u_{i,j+1} \} |.
\]
Plugging this discretization into \eqref{eq:pottsRadon},
we rewrite the  Potts problem as the constrained optimization problem
\begin{equation}
\begin{array}{lll}
		\displaystyle\minimize_{u_1, u_2, v}
  & \gamma (\| \nabla_{p_1} u_1\|_0 + \| \nabla_{p_2} u_2 \|_0)
  +   \|Av -f\|_2^2 \\[2ex]
  \text{subject to }&v-u_1 = 0,\quad v-u_2 = 0, \quad u_1-u_2 = 0.
\end{array}  
\end{equation}
The augmented Lagrangian of this optimization problem reads
\begin{multline}\label{eq:augLagAnIso}
L(u_1, u_2, v, \lambda_1, \lambda_2, \rho) =  \ \gamma  \, (\| \nabla_{p_1} u_1\|_0 + \| \nabla_{p_2} u_2 \|_0) + \| A v - f\|_2^2  \\
+ \langle \lambda_1, v - u_1 \rangle  + \tfrac{\mu}{2} \| v - u_1 \|_2^2 
 + \langle \lambda_2, v - u_2 \rangle + \tfrac{\mu}{2} \| v - u_2 \|_2^2
 \\
+ \langle \rho, u_1 - u_2 \rangle + \tfrac{\nu}{2} \| u_1 - u_2 \|_2^2.
\end{multline}
The constraints are now part of the (multivariate) target functional $L.$
The parameter $\nu > 0$ controls how strong 
the split variables $u_1, u_2$  are tied to each other 
and $\mu > 0$  controls their coupling to $v.$ 
The variables $\lambda_1,$ $\lambda_2,$ and $\rho$ are  $({m \times n})$-dimensional arrays of Lagrange multipliers. The inner product is defined as $\langle x, y \rangle = \sum_{i,j} x_{ij} y_{ij}.$ 
Completing the squares in \eqref{eq:augLagAnIso}, we reformulate $L$ in the convenient form
\begin{multline}\label{eq:augLagAnIso2}
L(u_1, u_2, v, \lambda_1, \lambda_2, \rho) =  \ \gamma  \, (\| \nabla_{p_1} u_1 \|_0 + \| \nabla_{p_2} u_2 \|_0) + \| A v - f\|_2^2 \\
+  \tfrac{\mu}{2} \| v - u_1 + \tfrac{\lambda_1}{\mu} \|_2^2  - \tfrac{\mu}{2}\|\tfrac{\lambda_1}{\mu}\|_2^2   
+ \tfrac{\mu}{2} \| v - u_2 + \tfrac{\lambda_2}{\mu} \|_2^2  - \tfrac{\mu}{2}\|\tfrac{\lambda_2}{\mu}\|_2^2 \\
+ \tfrac{\nu}{2} \| u_1 - u_2 + \tfrac{\rho}{\nu} \|_2^2  - \tfrac{\nu}{2}\|\tfrac{\rho}{\nu}\|_2^2.
\end{multline}
We now use the alternating direction method of multipliers (ADMM).
The basic idea of ADMM is to minimize the augmented Lagrangian $L$ 
  with respect to $u_1,$ $u_2,$ and $v$ separately and to perform gradient ascent steps with respect to the Lagrange multipliers. (We refer to \cite{boyd2011distributed} for a detailed treatment on optimization strategies based on ADMM.)
To simplify the expressions for  $\argmin_{u_s} L,$ $s=1,2,$ and $\argmin_{v} L,$ we will use the following lemma.
\begin{lemma}\label{lem:sumTerms}
For $a, b_1,...,b_N \in \R$ and $x_1,...,x_N > 0,$ we have that
\[ 
	\textstyle \sum_i x_i (a - b_i)^2 = (\sum_i x_i)  (a -   \frac{\sum_i b_i x_i}{\sum_i x_i})^2 + C 
\]
where   $C \in \R$ is a constant that does not depend on $a.$
 \end{lemma}
 \begin{proof} We calculate
 	\begin{align*}\textstyle
		\sum_i x_i (a - b_i)^2 &= \textstyle a^2 (\sum_i x_i) - 2a (\sum_i b_i x_i) + \sum_i  b_i^2x_i \\
		 &= \textstyle (\sum_i x_i) \paren{ a^2 - 2a \frac{\sum_i b_i x_i}{\sum_i x_i} + \frac{\sum_i  b_i^2x_i}{\sum_i x_i}} \\
		 &= \textstyle (\sum_i x_i) \paren{ a^2 - 2a \frac{\sum_i b_i x_i}{\sum_i x_i} + \paren{\frac{\sum_i b_i x_i}{\sum_i x_i}}^2 - \paren{\frac{\sum_i b_i x_i}{\sum_i x_i}}^2 - \frac{\sum_i  b_i^2x_i}{\sum_i x_i}} \\
		 &= \textstyle (\sum_i x_i) \paren{ \paren { a - \frac{\sum_i b_i x_i}{\sum_i x_i}}^2  - \paren{\frac{\sum_i b_i x_i}{\sum_i x_i}}^2 - \frac{\sum_i  b_i^2 x_i}{\sum_i x_i}}.
	\end{align*}
	The last two terms do not depend on $a,$ which shows the assertion.
 \end{proof}
 Using Lemma \ref{lem:sumTerms}, we rearrange the quadratic summands of
 $\argmin_{u_s} L$ for $s=1,2$ and those of $\argmin_{v} L.$ 
Doing so, we get the iteration
\begin{equation}\label{eq:anisoADMM}
\left\{
\begin{array}{rcl}
	u_1^{k+1} &\in& \argmin_{u_1} \frac{2\gamma}{\mu_k + \nu_k} \| \nabla_{p_1} u_1\|_0 + 
	\| u_1 - \frac{1}{\mu_k + \nu_k}  (\mu_k v^k + \nu_k u_2^k + \lambda_1^k - \rho^k ) \|_2^2, \\
		u_2^{k+1} &\in& \argmin_{u_2} \frac{2\gamma}{\mu_k + \nu_k} \| \nabla_{p_2} u_2\|_0 + 
	\| u_2 - \frac{1}{\mu_k + \nu_k}  ( \mu_k v^k + \nu_k u_1^{k+1} + \lambda_2^k + \rho^k )  \|_2^2, \\
		v^{k+1} &=& \argmin_v  \| A v - f\|_2^2 + \frac{\mu_k + \nu_k}{2} \| v - \frac{1}{2 \mu_k} ( \mu_k u_1^{k+1} +  \mu_k u_2^{k+1} - \lambda_{1}^k - \lambda_{2}^k)   \|_2^2 , \\
				\lambda_1^{k+1} &=& \lambda_1^k + \mu_k (v^{k+1}-u_1^{k+1}), \\
	\lambda_2^{k+1} &=& \lambda_2^k + \mu_k (v^{k+1}-u_2^{k+1}), \\
	\rho^{k+1} &=& \rho^k + \nu_k (u_1^{k+1}-u_2^{k+1}). \\
	\end{array}
	\right.
\end{equation}
As coupling parameter, we use an increasing sequence $(\mu_k)_{k\in \N}.$
This is a slight refinement of the standard ADMM \cite{wright2006numerical}.

The crucial observation is that we can solve all of the  subproblems of \eqref{eq:anisoADMM} 
efficiently.
The first line decomposes into $n$ univariate Potts problems of the form
\begin{equation}\label{eq:rowWise}
\begin{split}
(u_1^{k+1})_{:,j} 
&\in \argmin_{g \in \R^m} 
 \frac{2\gamma}{\mu_k + \nu_k}  \| \nabla g\|_0 + 
\| g -  \tfrac{1}{\mu_k + \nu_k}(\mu_k v^{k}_{:,j} + \nu_k (u_2^{k})_{:,j}  +  (\lambda_{1}^k)_{:,j} - \rho^k_{:,j}) \|_2^2,
\end{split}
\end{equation}
where we use the subscript notation $x_{:,j}$ to denote the $j$-th 
row of the $(m \times n)$-image $x,$ that is,
$x_{:,j} = (x_{ij})_{i=1,...,m}.$
Analogously, we get a  decomposition for the second line of \eqref{eq:anisoADMM} into the problems
\begin{equation}\label{eq:colWise}
\begin{split}
(u_2^{k+1})_{i,:} 
&\in \argmin_{g \in \R^n} 
 \frac{2\gamma}{\mu_k + \nu_k}  \| \nabla g\|_0 + 
\| g -  \tfrac{1}{\mu_k + \nu_k}(\mu_k v^{k}_{i,:} + \nu_k (u_1^{k+1})_{i,:}  +  (\lambda_{2}^k)_{i,:} + \rho^k_{i,:})  \|_2^2.
\end{split}
\end{equation}
The third  line of 
\eqref{eq:anisoADMM} 
is a classical $L^2$ Tikhonov regularization.
The last three lines are simple gradient-ascent steps in the Lagrange multipliers.
We briefly describe in Section \ref{sec:subproblems} the  strategies to solve these subproblems.

\subsection{Design of Isotropic Discretizations}

The anisotropic discretization of Section \ref{sec:basicSplitting}
 measures the length of the jump set 
in the anisotropic Manhattan metric \cite{chambolle1995image}.
This leads to geometric staircasing in the reconstructions
illustrated in Figure \ref{fig:isoVsAniso}.
The Euclidean length can be approximated better when complementing the neighborhood system with
finite-difference vectors,
for example, diagonal directions or \enquote{knight-move} directions \cite{chambolle1999finite}.
  We now present a general scheme to construct appropriate neighborhood systems.

 \begin{figure}
\def\thisfigwidth{0.23\textwidth}
\def\figfolder{experiments/isoVsAnisoGeo/}
\centering
\begin{subfigure}[t]{\thisfigwidth}
\includegraphics[width=\textwidth]{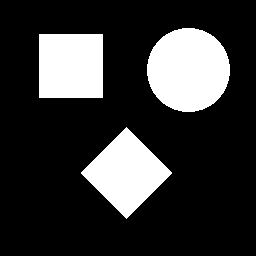} 
\caption{Original.}
\end{subfigure}
\hfill
\begin{subfigure}[t]{\thisfigwidth}
\includegraphics[width=\textwidth]{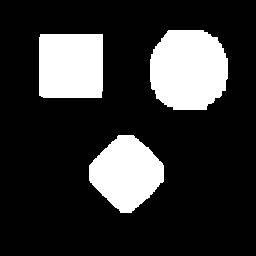} 
\caption{Reconstruction using $\Nc_0$ (PSNR: 
$\protect\input{\figfolder psnrPottsAniso.txt},$ MSSIM: $\protect\input{\figfolder ssimPottsAniso.txt}$).}
\end{subfigure}
\hfill
\begin{subfigure}[t]{\thisfigwidth}
\includegraphics[width=\textwidth]{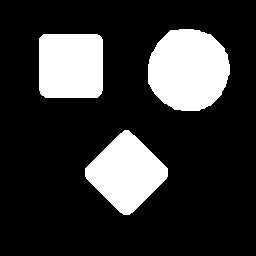}
\caption{Reconstruction using $\Nc_1$ (PSNR: 
$\protect\input{\figfolder psnrPottsIso1.txt},$ MSSIM: $\protect\input{\figfolder ssimPottsIso1.txt}$).}
\end{subfigure}
\hfill
\begin{subfigure}[t]{\thisfigwidth}
\includegraphics[width=\textwidth]{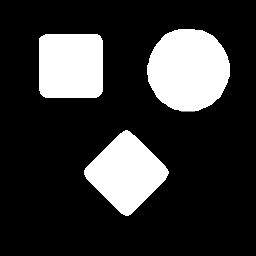}
\caption{Reconstruction using $\Nc_2$ (PSNR: 
$\protect\input{\figfolder psnrPottsIso2.txt},$ MSSIM: $\protect\input{\figfolder ssimPottsIso2.txt}$).}
\end{subfigure}
		\caption{Reconstruction of geometric shapes
		from Radon data with $\protect\input{\figfolder /nAngles.txt}$ angles and noise level $\protect\input{\figfolder /noiseLevel.txt}.$
		 In case of the anisotropic discretization $\Nc_0,$ jumps with respect to compass directions are
		 significantly less penalized than jumps with respect to the diagonal directions (see Figure \ref{fig:anisotropy}).
	 In consequence, horizontal and vertical edges are favored which results in a geometric staircasing effect.
		  The neighborhood $\Nc_1$ improves the result significantly.
The knight-move system $\Nc_2$ gives the most accurate reconstruction of the geometric shapes.
		}
		\label{fig:isoVsAniso}
\end{figure}

 The starting point is the anisotropic neighborhood system
 \[
 \Nc_0 = \{ (1,0), (0,1)\}.
 \]
 The vectors in this system have the (formal) slopes $0$ and $\infty.$
 We add a new finite-difference vector $(x,y) \in \Z^2$ to the system only if its slope $y/x$ is not yet contained in the system. 
 For example, we can add the vector $(1,1)$ with slope  $1.$
For reasons of symmetry, we also add the orthogonal vector $(1,-1).$
Thus, we get the neighborhood system
\begin{equation}\label{eq:N1}
	 \Nc_1 = \{ (1,0), (0,1), (1,1), (1,-1)\}.
\end{equation}
The next vectors to include in the neighborhood system are the
four knight move vectors
$(\pm 2,1),$ $(1,\pm 2)$ which leads to the system
\begin{equation}\label{eq:N2}
	 \Nc_2 = \{ (1,0), (0,1), (1,1), (1,-1), (-2,-1), (-2,1), (2,1), (2,-1)\}.
\end{equation}
The general scheme of adding new vectors corresponds to the standard enumeration of the rational numbers.

Appropriate weights can be derived as follows.
Let us assume that $u$ is a binary $(n\times n)$ image with an edge along the direction $(x,y) \in \Nc.$ 
We first look at lines with a slope $y/x$ between $(-1)$ and $1$ going  from the left to the right boundary of the image.
(If the slope of $(x,y)$ is not in the interval  $[-1,1]$ then we look at the $\pi/2$-rotated image and exchange the roles of $x$ and $y.$)
The Euclidean length of such a line is given by $n\sqrt{x^2 + y^2}/x.$
Since we want that the total jump length of this image equals that Euclidean length,
we get a condition on the weights that takes the form
\begin{equation}\label{eq:weightConditions}
\sum_{s=1}^S  \omega_{s}\, \| \nabla_{p_s} u \|_0 = n\frac{\sqrt{x^2 + y^2}}{x},
\end{equation}
where $\| \nabla_{p_s} u \|_0 $ is given by
\[
	\| \nabla_{p_s} u \|_0 = | \{  (i,j) :  u_{(i,j) +p_s} \neq u_{(i,j) }   \} |.
\]
It remains to evaluate the left-hand side of \eqref{eq:weightConditions}
for the binary image $u.$
This can be done either manually for small neighborhood systems
or with the help of a computer program for larger neighborhood systems.
When counting the non-zero entries of $\nabla_{p_s} u$ we assume $n$ to be large so that boundary effects are negligible.
We end up with a system of $S$ equations for the $S$ unknowns.
For the diagonal neighborhood system $\Nc_1,$ \eqref{eq:weightConditions}
yields the conditions
\begin{equation*}
\begin{alignedat}{8}
      \omega_1 &  &  & {}+{}     &   \omega_3 &{}+{}     &   \omega_4   & =  1,\\
       & {} {} & \omega_2 & {}+{}     &   \omega_3 &{}+{}     &   \omega_4   & =  1,\\
     \omega_1 &  {}+ {} &  \omega_2 & {}+{} & 2 \omega_3 &   & & = \sqrt{2}, \\
          \omega_1 &  {}+ {} &  \omega_2 &  &  & + & 2 \omega_4  & = \sqrt{2}.
\end{alignedat}
\end{equation*}
Solving this linear system, we get the weights
\[
	\omega_{1} = \omega_2 = \sqrt{2} - 1 \quad \text{and} \quad \omega_{3} = \omega_4 = 1 - \frac{\sqrt{2}}{2}.
\]
For the knight-move neighborhood system $\Nc_2,$
we get 
an analogous system of equations in $S = 8$ unknowns which gives us the
weights
\[ 
\omega_{s} = 
\begin{cases}
\sqrt{5} - 2, &\text{for } s=1,2,\\
\sqrt{5} - \frac{3}{2}\sqrt{2}, &\text{for } s=3,4, \\
\frac{1}{2}(1 + \sqrt{2} - \sqrt{5}), &\text{for } s=5,..., 8. \\
\end{cases}
\]

\begin{figure}
\def\thisfigscaling{0.43\textwidth}
\def\thisfigheight{0.9\textwidth}
\def\thisfigwidth{0.9\textwidth}
\centering
\begin{subfigure}[t]{\thisfigscaling}\centering
	\input{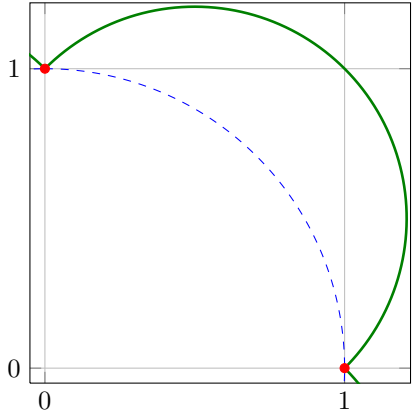}
	\caption{Anisotropic system $\Nc_0,$ ($E_0  \approx 1.41$).}
	\end{subfigure}
	\hspace{0.1\textwidth}
	\begin{subfigure}[t]{\thisfigscaling}\centering
		\input{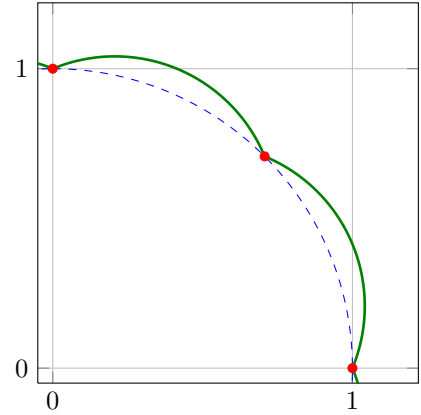}
			\caption{System with diagonals $\Nc_1$ ($E_1 \approx 1.08$).}
			\end{subfigure} \\[2ex]
	\begin{subfigure}[t]{\thisfigscaling}\centering
	\input{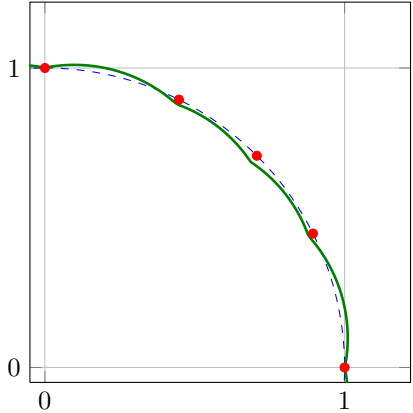}
	\caption{System with knight moves $\Nc_2$ using the weights of \cite{chambolle1999finite} ($E_2' \approx 1.05$).}
		\end{subfigure}
			\hspace{0.1\textwidth}
	\begin{subfigure}[t]{\thisfigscaling}\centering
	\input{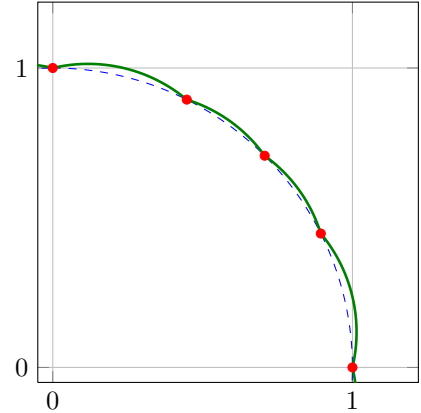}
	\caption{System with knight moves $\Nc_2$ using  our weights ($E_2 \approx 1.03$).}
		\end{subfigure}
		\caption{The solid line represents the length of 
		a Euclidean unit vector as measured in the finite-difference systems (as a function of the angle). The red dots identify the normalized vectors of the neighborhood system $p_s/\|p_s\|_2$ and the dashed line is the Euclidean unit circle. The isotropy increases  significantly when passing from $\Nc_0$ to $\Nc_1.$
		The increase in isotropy is less substantial when passing from $\Nc_1$ to $\Nc_2.$
		}
		\label{fig:anisotropy}
\end{figure}

We now turn to the question of how well we approximate the Euclidean length with the above discretizations.
The neighborhood systems give rise to a norm $\| \cdot \|_{\Nc}$ defined for $p \in \R^2$ by
\[
	\| p \|_{\Nc} = \sum_{s=1}^S \omega_s | \langle p, p_s \rangle |.
\]
By construction, the length 
$\|p_s \|_\Nc$ coincides with the Euclidean length $\| p_s\|_2$ for all vectors $p_s$ in the neighborhood system
as illustrated in Figure \ref{fig:anisotropy}.
In \cite{chambolle1999finite},
it is proposed to measure the isotropy of a finite-difference system
by the ratio $E$ between the longest and the shortest unit vector with respect to that length, compiled as
\[
	E = \max_{\|p\|_2 = 1} \| p \|_{\Nc}/\min_{\|p\|_2 = 1} \| p \|_{\Nc}.
\]
The closer the quantity $E$ is to one the higher is the isotropy.
For the anisotropic system  $\Nc_0,$ we get the value $E_0 = \sqrt{2} \approx 1.41.$
The introduction of diagonal directions reduces this value significantly to $E_1 \approx 1.08.$ 
If we include also the knight-move differences ($\Nc_2$), then the ratio improves further  to $E_2 \approx 1.03.$ 
We note that the weights 
for the system $\Nc_1$ coincide with those proposed in \cite{chambolle1999finite}, up to a normalization factor.
For $\Nc_2,$ our weights are more isotropic 
than the weights of \cite{chambolle1999finite}.

\subsection{Splitting Scheme for General Discretizations}
\label{sec:generalSplitting}

We now derive a minimization strategy for 
the general discretization \eqref{eq:pottsWeight}.
Let us denote the neighborhood system by
$\Nc = \{ p_1,...,p_S\}$ and let $\omega_1,..., \omega_S > 0$ where $S \geq 2.$
We first rewrite \eqref{eq:pottsWeight} as the constrained optimization problem
\begin{equation}
\begin{split}
	&\minimize_{u_1, ..., u_S, v} \quad \gamma \sum_{ s=1}^S \omega_s \| \nabla_{p_s} u_s \|_0 + \| A v - f\|_2^2 \\
		&\text{subject to} \quad 
	u_r - u_t = 0, \text{  for all }1 \leq r < t \leq S, \\
	&\text{\phantom{subject to}} \quad 
		v - u_s  = 0, \text{  for all }1 \leq s \leq S. 
	\end{split}
\end{equation}
The augmented Lagrangian of this optimization problem reads
\begin{multline}
L(u, \lambda, \rho) = \gamma \sum_{s=1}^S \omega_s \| \nabla_{p_s} u_s \|_0  + \frac{\mu}{2} \| v - u_s + \frac{\lambda_{s}}{\mu} \|_2^2 \\
	+ \frac{\nu}{2}\sum_{1 \leq r < t \leq S}\| u_r - u_t + \frac{\rho_{r,t}}{\nu} \|_2^2 
	+ \| A v - f\|_2^2,
\end{multline}
where $L$ depends on the variables 	$\{u_s\}_{1 \leq s \leq S},$ $\{\lambda_s\}_{1 \leq s \leq S}$ and $\{\rho_{r,t}\}_{1 \leq r < t \leq S}$.
The parameter $\nu > 0$ controls how strong 
the split variables $u_1, ..., u_S$  are tied to each other 
and $\mu > 0$  controls their coupling to $v.$ 
The variables $\lambda_{s}, \rho_{r,t} \in \R^{n\times m}$ are Lagrange multipliers.
In the ADMM iteration, we minimize $L$ sequentially 
with respect to $v, u_1, ..., u_S$ followed by a gradient ascent step in the Lagrange multipliers.
The minimization of $L$ with respect to $u_s$ reads
\begin{multline}
	\argmin_{u_s} L(u, \lambda, \rho) = \argmin_{u_s} \gamma \omega_s  \| \nabla_{p_s} u_s \|_0 + 
	\frac{\mu}{2} \| v - u_s  + \tfrac{\lambda_{s}}{\mu} \|_2^2 \\
	+ \frac{\nu}{2}\sum_{1 \leq r <s }\| u_r - u_s + \tfrac{\rho_{r,s}}{\nu} \|_2^2
	+ \frac{\nu}{2}\sum_{s < t \leq S }\| u_s - u_t + \tfrac{\rho_{s,t}}{\nu} \|_2^2. 
\end{multline}
We modify this expression using Lemma \ref{lem:sumTerms} to 
\begin{equation*}
	\argmin_{u_s} L(u, \lambda, \rho) = \argmin_{u_s} \frac{2\gamma \omega_s}{\mu+\nu (S-1)}  \| \nabla_{p_s} u_s \|_0 
	+ \| u_s -  w_s\|_2^2 
\end{equation*}
with 
\[
 w_s = 	 \frac{ \mu v + \lambda_s +  \sum_{1 \leq r < s} ( \nu u_r + \rho_{r,s}) + \sum_{s < t \leq S} (\nu u_t - \rho_{s,t})}{\mu + \nu (S-1)} .
\]
In a similar way,  we consider the minimizer with respect to $v$ as
\[
	\argmin_{v} L(u, \lambda, \rho) = \argmin_{v}  \| A v - f\|_2^2 + \sum_{s=1}^S \frac{\mu}{2} \| v - u_s + \frac{\lambda_{s}}{\mu} \|_2^2
	 \]
	 which we rewrite using Lemma \ref{lem:sumTerms} as 
\[
\argmin_{v} L(u, \lambda, \rho) = \argmin_{v}  \| A v - f\|_2^2
	+ \frac{\mu S}{2} \| v -  \frac{1}{S} \sum_{s = 1}^S ( u_s - \frac{\lambda_s}{\mu})  \|_2^2.
\]
Having computed explicit expressions for the minimization with respect to each variable, we obtain the ADMM iteration 
\begin{equation}\label{eq:generalADMM}
\left\{
	\begin{array}{rcl}
		u_1^{k+1} &\in& \argmin_{u_1} \frac{2\gamma \omega_1}{\mu_k + \nu_k (S-1)}  \| \nabla_{p_1} u_1 \|_0 
	+\| u_1 -  w_1^k \|_2^2, \\[2ex]
	&\vdots&
	\\[2ex]
	u_S^{k+1} &\in& \argmin_{u_S}  \frac{2\gamma \omega_S}{\mu_k +\nu_k (S-1)}  \| \nabla_{p_S} u_S \|_0 + \| u_S -  w_S^k \|_2^2, \\[2ex]
		v^{k+1} &=& \argmin_{v}  \| A v - f\|_2^2
	+  \frac{\mu_k S}{2}\| v -  z^k \|_2^2, 
  \\[2ex]
  		\lambda_{s}^{k+1} &=& \lambda_{s}^{k} + \mu_k (v^{k+1} - u_s^{k+1}), \quad\text{for all}\quad 1\leq s \leq S, \\[2ex]
		\rho_{r,t}^{k+1} &=& \rho_{r,t}^{k} + \nu_k (u_r^{k+1} - u_t^{k+1}), \quad\text{for all}\quad 1\leq r < t \leq S.
	\end{array}
	\right.
\end{equation}
Here, $w_s^k$ is given by 
\[
 w_s^k = \frac{ \mu_k v^{k} + \lambda_s^{k} +  \sum_{1 \leq r < s} ( \nu_k u_r^{k+1} +\rho_{r,s}^{k}) + \sum_{s < t \leq S} (\nu_k u_t^{k} - \rho_{s,t}^{k})}{\mu_k + \nu_k (S-1)}
\]
and $z^k$ by
\[
	z^k = \frac{1}{S} \sum_{s = 1}^S \paren{  u_s^{k+1} - \frac{\lambda_s^{k}}{ \mu_k}}. 
\]
The key observation is that every of the subproblem in the ADMM iteration
can be solved efficiently.
The minimization problems with respect to $u_1,...,u_S$
decompose into univariate Potts problems 
with respect to the paths induced by the finite difference vectors $p_s.$
To fix ideas, consider the finite-difference vector $p_s = (1,1).$ Then, the solution $u^{k+1}_s$ 
is given by solving one univariate Potts problem for each diagonal path of the two-dimensional array $w_s^k.$
The last subproblem, like in the anisotropic case, is a classical Tikhonov-type regularization.

We eventually remark that the anisotropic splitting \eqref{eq:anisoADMM} 
is a special case of the general form \eqref{eq:generalADMM} 
if we choose the anisotropic finite-difference system $\Nc_0.$

\subsection{Solution of the Subproblems}
\label{sec:subproblems}
Solution methods for the univariate
 Potts and classical Tikhonov problems are well studied.
Since they are the important building blocks of our iteration \eqref{eq:generalADMM},
we  briefly recall the idea of the algorithms.

The classical univariate Potts problem is given by
\begin{equation}\label{eq:potts1D}
	P_\gamma(g) = \gamma \, \| \nabla g \|_0 + \| g - f\|_2^2 \to \mathrm{min},
	\end{equation}
	where $g,f \in \R^{n}$ and $\| \nabla g \|_0 = | \{ i : g_{i} \neq g_{i+1}\}|$ denotes the number of jumps of $g.$
	This  can be solved exactly by dynamic programming \cite{chambolle1995image, friedrich2008complexity, mumford1985boundary,mumford1989optimal,    storath2014fast, demaret2012reconstruction, winkler2002smoothers}.
 The basic idea is that a minimizer of the Potts functional
for data $(f_1,...,f_{r})$ can be computed in polynomial time provided that minimizers of the  partial data $(f_1),$ $(f_1, f_2),$ $...,$ $ (f_1, ..., f_{r-1})$ are known.
We denote the respective minimizers by $g^1,$  $g^2,$ ..., $g^{r-1}.$
In order to compute a minimizer for data $(f_1,...,f_{r}),$ 
we first create a set of $r$ minimizer candidates $h^1,$ ..., $h^{r},$ each of length $r.$
These minimizer candidates are given by
\begin{equation}\label{eq:potts_candidate}
	h^\ell = (g^{\ell-1} , \underbrace{\mu_{[\ell,r]},..., \mu_{[\ell,r]}}_{\text{Length } (r - \ell +1)}), \\
\end{equation}
where $g^0$ is the empty vector and 
$\mu_{[\ell, r]}$ denotes the mean value of data $f_{[\ell, r]} = (f_{\ell}, ..., f_{r}).$
 Among the candidates $h^\ell,$ one with the least Potts functional value is a minimizer for the data $f_{[1, r]}.$

In \cite{friedrich2008complexity}, Friedrich et al.~proposed the following $\Oc(n^2)$ time and $\Oc(n)$ space algorithm.
They observed that the functional values of a minimizer $P_\gamma(g^r)$ for data 
$f_{[1,r]}$ can be computed directly from the functional values $P_\gamma(g^1),$ ..., $P_\gamma(g^{r-1})$ 
and the squared mean deviations of the data
$f_{[1, r]},$ ..., $f_{[r, r]}.$ 
Indeed, using \eqref{eq:potts_candidate}, the functional value of the minimizer $g^r$ is given (setting  $P_\gamma(g^0) = -\gamma$) by 
\begin{equation}\label{eq:potts_value_candidate}
\begin{split}
P_\gamma(g^r) =& \min_{\ell=1,...,r} \{  P_{\gamma}(g^{\ell-1})  +  \gamma 
+ 	d_{[\ell, r]} \},
\end{split}
\end{equation} 
where  $d_{[\ell, r]}$ denotes the squared deviation from the mean value
\[
	d_{[\ell, r]} = \min_{y \in \R} \| y - f_{[\ell,r]}\|^2_2 = \| (\mu_{[\ell, r]},...,\mu_{[\ell, r]})  - f_{[\ell,r]}\|^2_2.
\]
The evaluation of \eqref{eq:potts_value_candidate} is $\Oc(n)$ if we precompute  the first and second moments of data $f_{[\ell,r]}.$
If $\ell^{\ast}$ denotes the minimizing argument in \eqref{eq:potts_value_candidate}, then $(\ell^{\ast} -1)$ indicates the rightmost jump location at step $r,$ which is stored as $J(r).$ 
The jump locations of a solution $g^r$ are thus $J(r),$ $J(J(r)),$ $J(J(J(r))),...;$
the values of $g^r$ between two consecutive jumps are given by the  mean  of  $f$ on this interval. 
Note that  we only have to compute and  store the jump locations $J(r)$ and the minimal Potts functional value $P_\gamma(g^r)$ in each iteration. 
The reconstruction of the minimizer from the jump locations only has to be done once for $g^n$ at the end; it is thus not time-critical. 

The algorithm for solving \eqref{eq:potts1D} consists of two nested loops for $r =1,...,n$ and $\ell = 1,...,r,$ which amounts to $n(n+1)/2$ iterations in total.
Typically, a significant amount of configurations are unreachable
and thus can be skipped \cite[Theorem 2]{storath2014fast}.
The time complexity is still $\Oc(n^2),$ but the practical runtime 
is improved 
by a fourfold to fivefold factor.
We refer the reader to \cite{storath2014fast} for the complete flow diagram of the accelerated algorithm.
 We remark that a minimizer of the univariate Potts problem 
 need not be unique, which explains the \enquote{$\in$} in \eqref{eq:anisoADMM}, \eqref{eq:rowWise}, \eqref{eq:colWise}, and \eqref{eq:generalADMM}.
 However, those data
  which lead a non-unique minimizer form a negligible set
 \cite{wittich2008complexity}.

Our second subproblem is the solution of a classical $L^2$ Tikhonov regularization
\begin{equation}\label{eq:tikhonov}
 v^{k+1} = \argmin_{v}  \|A v - f\|_2^2 + \frac{\mu}{2} \| v - z  \|_2^2
\end{equation}
with some $z \in \R^{m\times n}.$
The unique minimizer of this problem is given by the solution of the normal equation 
\begin{equation}\label{eq:normalEq}
	(A^* A + \frac{\mu}{2} I) \,v = A^* f + \frac{\mu}{2}z.
\end{equation}
Here, $A^\ast$ is the adjoint of $A.$
This linear system can be solved using, for example, the conjugate-gradient method. In some cases we can exploit the structure of $A$ 
 for more efficient solution methods. 
 This is the case when $A$ is the Radon transform or a convolution operator (see Section \ref{sec:applications}). 

\subsection{General Data Terms}
Inspecting the ADMM iteration, 
we observe that the data term only appears in the first line.
That line consists of a classical Tikhonov regularization with the $L^2$ data term
$\|A v - f\|_2^2.$ 
Minimizers of that problem 
can be computed efficiently for many other data terms $d(u, f),$
such as $L^p$ data terms, $p \geq 1,$ of the form
\[
	d(v, f) = \| Av - f \|_p^p
\]
or a Huber data term which is a hybrid between $L^1$ and $L^2$ data terms \cite{chambolle2011first, weinmann2013total}.
In general, our algorithm is applicable whenever the proximity operator $\prox_{d(\cdot, f)/\mu}$ of $d(\cdot,f),$
defined by
\begin{equation*}
 	\prox_{d(\cdot, f)/\mu }(z) = \argmin_{v}  d(v, f) + \frac{\mu}{2} \| v - z  \|_2^2,
\end{equation*}
can be evaluated efficiently.
This is often the case when $d(\cdot, f)$ is a convex functional.
We refer to \cite{boyd2004convex} for an extensive overview on strategies for convex optimization.

\subsection{Convergence}

In this section, we show that Algorithm \eqref{eq:generalADMM} 
 converges
in the prototypical case $\nu_k = 0$  for all $k$
(which implies that the $\rho^k = 0$ for all $k$.)
We leave a convergence proof for Algorithm \eqref{eq:generalADMM} 
with general $\nu_k$ as an open problem.
For the proof, we use methods developed  in \cite{storath2013jump}.

\begin{theorem}\label{thm:convergence}
Let the sequence $(\mu_k)_{k\in \N}$ be increasing and satisfy $\sum_k \mu_k^{-1/2} < \infty.$ 
Further, let $\nu_k=0$ for all $k.$
Then, the iteration \eqref{eq:generalADMM} converges in the sense that
\begin{align}
(u_1^k,\ldots, u_S^k,v^k) &\to (u_1^*,\ldots, u_S^*,v^*) \quad \text{ with } \quad u_1^* = \ldots = u_S^* = v^*,  \notag \\
\tfrac{\lambda_s^k}{\mu_k} &\to 0  \quad  \text{ for all } \quad s \in \{1,\ldots,S\}.  \label{eq:whatWSh}
\end{align}
\end{theorem}

\begin{proof}
We denote the $S$ functionals appearing in the first $S$ lines of \eqref{eq:generalADMM} by $F^{k}_s,$ i.e.,
$$
   F^{k}_s (u_s) = \frac{2\gamma \omega_s}{\mu_k}  \| \nabla_{p_s} u_s \|_0   +\| u_s -  (v^{k} + \frac{\lambda_s^k}{\mu_k}) \|_2^2.
$$
Using this notation, we  rewrite the first $S$ lines of \eqref{eq:generalADMM} as $u_s^{k+1} \in \argmin_u F^{k}_s (u_s)$ 
for all $s \in \{1,\ldots,S\}.$
We first estimate the distance $\|u_s^{k+1} - (v^k + \tfrac{\lambda_s^k}{\mu_k})\|_2.$ To that end, we note that
$
 F^{k}_s (u_s^{k+1}) \leq F^{k}_s \left(v^k+ \tfrac{\lambda_s^k}{\mu_k} \right)
$
which holds true since $u_s^{k+1}$ minimizes $F^{k}_s.$
Applying the definition of $F^{k}_s,$ we get
\begin{align*}
   \gamma \omega_s \|\nabla_{p_s} u_s^{k+1}\|_0  +& \frac{\mu_k}{2} \|u_s^{k+1} - \left(v^k + \frac{\lambda_s^k}{\mu_k} \right)\|_2^2 
   \leq  \gamma \omega_s \|\nabla_{p_s} \left(v^k + \frac{\lambda_s^k}{\mu_k}\right)\|_0 \leq \gamma \omega_s L,
\end{align*}
where 
$L= N M$ is the size of the considered $(N \times M)$ image.
This is because $ \|\nabla_{p_s} z \|_0 \leq M N$ for any $p_s$ and any data $z.$ 
Since the first summand on the left-hand side is nonnegative, we get that
\begin{equation}\label{eq:est1}
\|u_s^{k+1} - \left(v^k + \frac{\lambda_s^k}{\mu_k} \right)\|_2^2 \leq \frac{\gamma \omega_s L}{\mu_k}.
\end{equation}
In particular, for all $s \in \{1,\ldots,S\},$ we obtain that
\begin{equation} \label{eq:lim1}
   \lim_{k \to \infty} u_s^{k+1} - \left( v^k + \frac{\lambda_s^k}{\mu_k}   \right) =0.
\end{equation}

We now draw our attention to the $(S+1)$th line of \eqref{eq:generalADMM}. We denote the corresponding functional by $G^{k},$
i.e.,
$$
   G^{k}(v) =  \| A v - f\|_2^2
               	+ \frac{\mu_k S}{2} \| v -  \frac{1}{S}  \sum_{s=1}^S (u^{k+1}_s - \frac{\lambda_s^{k}}{\mu_k} )\|_2^2. \\
$$
The minimality of $v^{k+1}$ implies the inequality
$$
G^k(v^{k+1}) \leq G^k  \left(  \frac{1}{S}  \sum_{s=1}^S (u^{k+1}_s - \frac{\lambda_s^{k}}{\mu_k})      \right).
$$
We now apply the definition of $G^k$ to estimate 
\begin{align}
    \|Av^{k+1} - f\|_2^2 &+ \frac{\mu_k S}{2} \|v^{k+1} - \left( \frac{1}{S} \sum_{s=1}^S (u^{k+1}_s - \frac{\lambda_s^{k}}{\mu_k})   \right)\|_2^2 \notag \\ 
    \leq & \|A    \left(  \frac{1}{S}\sum_{s=1}^S (u^{k+1}_s - \frac{\lambda_s^{k}}{\mu_k})   \right) - f\|_2^2 \notag\\
    \leq &\|A \left(  \frac{1}{S}\sum_{s=1}^S (u^{k+1}_s - \frac{\lambda_s^{k}}{\mu_k} - v^k)  \right) + A v^k -f\|_2^2  \notag\\
    \leq & \left( \|A\| \| \frac{1}{S}\sum_{s=1}^S (u^{k+1}_s - \frac{\lambda_s^{k}}{\mu_k} - v^k) \|_2 +\|Av^k - f\|_2 \right)^2.  \label{eq:l3}
\end{align}
Here, $\|A\|$ is the operator norm of $A$ acting on $\ell^2$.
We now combine \eqref{eq:l3} and $\eqref{eq:est1}$ to estimate the magnitude of the residuals $Av^{k+1} - f.$  We get
$$
    \|Av^{k+1} - f\| \leq \frac{C}{\sqrt{\mu_{k} }} + \|Av^{k} - f\|, 
$$
where $C>0$ is a constant that only depends on $\gamma, \omega_s$ $L,$ and $\|A\|.$
Solving this recursion yields
$$
    \|Av^{k+1} - f\| \leq  C \sum_{j=1}^k \frac{1}{\sqrt{\mu_j}} + \|A v^0 -f\|,
$$
which shows that the sequence of residuals $(Av^{k+1} - f)_{k \in \N}$ is bounded.

We consider the right-hand term in the first line of \eqref{eq:l3}. Then we apply \eqref{eq:l3} to get
\begin{align*}
\tfrac{\mu_k S}{2} \|v^{k+1} &-  \frac{1}{S}\sum_{s=1}^S (u^{k+1}_s - \frac{\lambda_s^{k}}{\mu_k})   \|^2 \\
&\leq    (\|A\| \|\frac{1}{S}\sum_{s=1}^S (u^{k+1}_s - \frac{\lambda_s^{k}}{\mu_k} - v^k)\| + \|Av^k - f\|)^2  \\
&\leq    (\|A\| \frac{1}{S}\sum_{s=1}^S \|(u^{k+1}_s - \frac{\lambda_s^{k}}{\mu_k} - v^k)\| + C')^2.
\end{align*}
The last inequality is a consequence of the boundedness of the residuals 
where we denote the bound by the positive constant $C'$ (which is independent of $k$). 
We now apply \eqref{eq:lim1} to the first summand
to conclude that the sequence (with respect to $k$)
\begin{equation} \label{eq:est2}
\mu_k \|v^{k+1} -    \frac{1}{S}\sum_{s=1}^S (u^{k+1}_s - \frac{\lambda_s^{k}}{\mu_k}) \|_2^2 \text{ is bounded}.
\end{equation}

We use this fact to establish the convergence of the sequence $v^k$ by showing that it is a Cauchy sequence. 
We first apply the triangle inequality to get
$$
   \|v^{k+1} - v^k\| \leq \|v^{k+1} -  \frac{1}{S}\sum_{s=1}^S (u^{k+1}_s - \frac{\lambda_s^{k}}{\mu_k} )   \| + 
   \| \frac{1}{S}\sum_{s=1}^S (u^{k+1}_s - \frac{\lambda_s^{k}}{\mu_k} - v^k)\|.
$$
We now apply \eqref{eq:est2} to the first summand on the right-hand side    
as well as \eqref{eq:est1} to the second summand on the right-hand side to obtain 
$$
    \|v^{k+1}-v^k\| \leq \frac{C''}{\sqrt{\mu_k}}
$$
for some constant $C'' >0$ which is again independent of $k$.
The assumption on the sequence $\mu_k$ guarantees that $v^k$ is a Cauchy sequence and hence that $v^k$ converges to some $v^*$.

To establish the last statement in \eqref{eq:whatWSh},
we rewrite each of the last $S$ lines in \eqref{eq:generalADMM} to obtain the identity
\begin{equation}\label{eq:thirdLineConsequ}
   \frac{\lambda_s^{k+1}}{\mu_{k}} =  (\frac{\lambda_s^k}{\mu_k} +  u_s^{k+1} -v ^k)+(v^k-v^{k+1}).
\end{equation}
By \eqref{eq:lim1} and \eqref{eq:est2}, each term in parenthesis converges to $0.$ 
Hence,
$$
   \lim_{k \to \infty} \frac{\lambda_s^{k+1}}{\mu_k} = 0.
$$
Since we assume that the sequence $\mu_k$ is nondecreasing, we have that
$\mu_k/\mu_{k+1} \leq 1$ and thus, for all $s= 1,\ldots,S$, 
$$
   \lim_{k \to \infty} \frac{\lambda_s^k}{\mu_k} = 0.
$$
This shows the last statement in \eqref{eq:whatWSh}. Finally, we rewrite the penultimate line of \eqref{eq:generalADMM} as $v^{k+1} - u_s^{k+1}$ $= (\lambda_s^{k+1} - \lambda_s^{k})/\mu_k$  
to obtain the inequality
$$
   \| u_s^{k+1} - v^{k+1} \| \leq \tfrac{\|\lambda_s^{k+1}\|} {\mu_k} + \tfrac{\|\lambda_s^{k}\|} {\mu_k} 
   \to 0.  
$$
This means that  $u_s^k - v^k \to 0$ for all $s= 1,\ldots,S $ and, since $v^k$ converges,
also each $u^k_s$ converges and the corresponding limit $u_s^*$ equals $v^*,$ which completes the proof.
\end{proof}

\section{Application to  Radon Data}
\label{sec:applications}

The result of our method is
a joint reconstruction and segmentation of the imaged object. 
More precisely, we obtain a piecewise constant image
which induces a partition on the image domain.
We demonstrate the applicability to
tomographic problems 
whose image acquisition process 
can be described in terms of the classical Radon transform \cite{Natterer86}.
We here consider x-ray computed tomography (CT) and positron emission tomography (PET). 
Recall that the Radon transform is defined by 
\begin{equation}\label{eq:radon}
	\Rc u(\theta,s) = \int_{-\infty}^\infty u(s\theta + t\theta^\bot)\d t,
\end{equation} 
where $s\in \R,$  $\theta\in S^1$, and $\theta^\bot$ is the unit
vector $\pi/2$ radians counterclockwise from $\theta.$

% !TEX root =./inversePotts2D.tex
\begin{figure}
\def\figfolder{newExperiments/radonAnglesNew/nAngles}
\def\figwidth{0.23\textwidth}
\footnotesize
\centering
\def\nAnglesA{17}
\def\nAnglesB{12}
\def\nAnglesC{7}
\def\figfolderA{\figfolder\nAnglesA/}
\def\figfolderB{\figfolder\nAnglesB/}
\def\figfolderC{\figfolder\nAnglesC A2/}
\begin{tabular}[t]{p{0.13\textwidth}rrr}
%{\textwidth}
				\toprule 
			%\multicolumn{3}{c}{Noise level}\\
 & \normalsize ${\input{\figfolderA/nAngles.txt}}$ Angles
& \normalsize ${\input{\figfolderB/nAngles.txt}}$ Angles
& \normalsize ${\input{\figfolderC/nAngles.txt}}$  Angles \\
%Noise level \nAnglesA &  Noise level \nAnglesB & Noise level \nAnglesC \\
   \toprule
   \normalsize FBP (Ram-Lak filter) &	 \includegraphicstotab[width=\figwidth]{\figfolderA /recFBPRamLak} &
   			\includegraphicstotab[width=\figwidth]{\figfolderB /recFBPRamLak} &
			\includegraphicstotab[width=\figwidth]{\figfolderC /recFBPRamLak}\\
	&	PSNR: ${\input{\figfolderA/psnrFBPRamLak.txt}}$ &
  PSNR: ${\input{\figfolderB/psnrFBPRamLak.txt}}$ &
  PSNR: ${\input{\figfolderC/psnrFBPRamLak.txt}}$ \\	
  & MSSIM: ${\input{\figfolderA/ssimFBPRamLak.txt}}$ & 
  MSSIM: ${\input{\figfolderB/ssimFBPRamLak.txt}}$ & 
  MSSIM: ${\input{\figfolderC/ssimFBPRamLak.txt}}$ \\
  	\midrule
\normalsize \vspace{-4ex} FBP (Hamming window, optimized to PSNR) 
&	 \includegraphicstotab[width=\figwidth]{\figfolderA /recFBP} &
   			\includegraphicstotab[width=\figwidth]{\figfolderB /recFBP} &
			\includegraphicstotab[width=\figwidth]{\figfolderC /recFBP}\\
	&	PSNR: ${\input{\figfolderA/psnrFBP.txt}}$ &
  PSNR: ${\input{\figfolderB/psnrFBP.txt}}$ &
  PSNR: ${\input{\figfolderC/psnrFBP.txt}}$ \\	
  & MSSIM: ${\input{\figfolderA/ssimFBP.txt}}$ & 
  MSSIM: ${\input{\figfolderB/ssimFBP.txt}}$ & 
  MSSIM: ${\input{\figfolderC/ssimFBP.txt}}$ \\
  	\midrule
  \normalsize TV & \includegraphicstotab[width=\figwidth]{\figfolderA/recTV} &
	\includegraphicstotab[width=\figwidth]{\figfolderB/recTV} &
	\includegraphicstotab[width=\figwidth]{\figfolderC/recTV} \\
&   $\alpha = {\input{\figfolderA/lambda.txt}}$ &
    $\alpha = {\input{\figfolderB/lambda.txt}}$ & 
    $\alpha = {\input{\figfolderC/lambda.txt}}$  \\
	&  PSNR:	 ${\input{\figfolderA/psnrTV.txt}}$ &
  PSNR: ${\input{\figfolderB/psnrTV.txt}}$ &
  PSNR: ${\input{\figfolderC/psnrTV.txt}}$ \\
  & MSSIM: ${\input{\figfolderA/ssimTV.txt}}$ & 
  MSSIM: ${\input{\figfolderB/ssimTV.txt}}$ & 
  MSSIM: ${\input{\figfolderC/ssimTV.txt}}$ \\
  \midrule
 \normalsize Our method & \includegraphicstotab[width=\figwidth]{\figfolderA/recPotts} &
	\includegraphicstotab[width=\figwidth]{\figfolderB/recPotts} &
	\includegraphicstotab[width=\figwidth]{\figfolderC/recPotts} \\
		&   $\gamma = {\input{\figfolderA/gamma.txt}}$ &
    $\gamma = {\input{\figfolderB/gamma.txt}}$ & 
    $\gamma = {\input{\figfolderC/gamma.txt}}$  \\
	&  PSNR:	 ${\input{\figfolderA/psnrPotts.txt}}$ &
  PSNR: ${\input{\figfolderB/psnrPotts.txt}}$ &
  PSNR: ${\input{\figfolderC/psnrPotts.txt}}$ \\
  & MSSIM: ${\input{\figfolderA/ssimPotts.txt}}$ & 
  MSSIM: ${\input{\figfolderB/ssimPotts.txt}}$ & 
  MSSIM: ${\input{\figfolderC/ssimPotts.txt}}$ \\
    & RI: ${\input{\figfolderA/randIndexPotts.txt}}$ & 
  RI: ${\input{\figfolderB/randIndexPotts.txt}}$ & 
  RI: ${\input{\figfolderC/randIndexPotts.txt}}$ \\
  \bottomrule
	\end{tabular}
	\caption{
	Reconstruction of the Shepp-Logan phantom  from highly undersampled Radon data. 
	Filtered backprojection produces strong artifacts.
	Total variation regularization gives an almost perfect reconstruction up to about 17 projection angles 
	but the quality decreases significantly for fewer angles.
	The proposed Potts based method yields a high quality reconstruction/segmentation from as few as $7$ projection angles.
	}
\label{fig:radonAnglesNew}
\end{figure}

\paragraph{Measurement of the segmentation quality.}
We will focus on data from piecewise-constant images.
This has the advantage 
that the ground truth, i.e. the desired partitioning, 
is induced directly by the original image.
(For natural images, there is typically no ground truth available
because segmentation is based on subjective impressions.)
Having a ground truth at hand, we can objectively measure the quality of the segmentation
using the Rand index (RI) \cite{rand1971objective, arbelaez2011contour}, 
which we briefly explain.
Let $X =\{x_1,...,x_N\}$ be a given set of points 
and let $Y$ and $Y'$ be two partitionings of the this set. 
(In our case, $X$ is the set of pixels, and $Y$ and $Y'$ are the partitioning 
given by the ground truth and the result of our method, respectively.)
The Rand index (RI) is defined by
\[\textstyle
	\RI(Y, Y') = \binom{N}{2}\sum_{i < j}^N t_{ij}
\]
where $t_{ij}$ is equal to one 
if there exist $k$ and $k'$ such that both $x_i$ and $x_j$ are in both $Y_k$ and $Y'_{k'},$
or if $x_i$ is in both $Y_k$ and $Y'_{k'}$ while $x_j$ is in neither $Y_k$ or $Y'_{k'}.$
The Rand index is bounded from above by $1;$  a higher value means a better match.
For the evaluation of the Rand index we used the implementation of K.~Wang available at the Matlab File Exchange.

\paragraph{Parameter choice for the algorithm.}
Unless stated otherwise, the setup for the numerical experiments is as follows:
 We use the ADMM iteration \eqref{eq:generalADMM} with the  coupling sequence $\mu_k = 10^{-7} \cdot k^{\tau}$  with $\tau = 2.01,$ and we choose $\nu_k$ identically zero. This choice satisfies the hypothesis of Theorem \ref{thm:convergence}.
For the experiments involving the Shepp-Logan phantom, we observed the best results using the neighborhood system 
$\Nc_1$ of \eqref{eq:N1} while   the neighborhood $\Nc_2$ of \eqref{eq:N2} gave better result for the more realistic images (Figures \ref{fig:PET}, \ref{fig:gaussianblur}, and \ref{fig:motionblur}); see also the comparison in Figure \ref{fig:isoVsAniso}.
For the solution of the Tikhonov type problem \eqref{eq:tikhonov},
we use Matlab's conjugate-gradient method on the normal equation \eqref{eq:normalEq}.
We use a \enquote{warm start}, which means that we use  the solution of the previous iteration $v^k$
as initial guess for $v^{k+1}.$ Then the conjugate-gradient iteration converges 
typically in few steps.  
The splitting variables $v,$ $u_s$ for $s \geq 1,$ and the Lagrange multipliers are all initialized with $0.$
We stop the iteration when the relative deviation of $u_1$  and $u_2,$ i.e.
$\|u_1 - u_2\|_2 / (\|u_1 \|_2 + \| u_2\|_2),$
 falls below some threshold.
In our experiments, we have chosen the tolerance $10^{-3};$  
 we did not observe an improvement of the results for lower thresholds.
  The distance of $u_1$ and $u_2$ is a natural choice because these variables
 appear in all the neighborhood systems including the anisotropic system. 
  For larger neighborhoods, other choices are possible,
 because  all $u_s$ converge to the same limit by Theorem \ref{thm:convergence}. 
Although the method is not independent 
from the ordering of the vectors in the neighborhood system, 
we observed no significant difference for other orderings.
In our simulated experiments, the noise is Gaussian distributed with zero mean and standard deviation 
$\sigma = \text{noiselevel} \cdot \| f'\|_\infty,$ where $f'$ is the clean data.

\subsection{Radon Data with Sparse Angular Sampling}

\begin{figure}[t]

\def\figwidth{0.23\columnwidth}
\def\figspace{\hspace{0.1\textwidth}}
\centering
\def\figfolder{newExperiments/pottsGammaInfluence12/}
	\def\figfolderA{newExperiments/pottsGammaInfluence12/gamma0.05/}
		\def\figfolderB{newExperiments/pottsGammaInfluence12/gamma0.2/}
			\def\figfolderC{newExperiments/pottsGammaInfluence12/gamma0.5/}
		\def\figfolderD{newExperiments/pottsGammaInfluence12/gamma1/}
	\begin{subfigure}[t]{\figwidth}
	\includegraphics[width=\columnwidth]{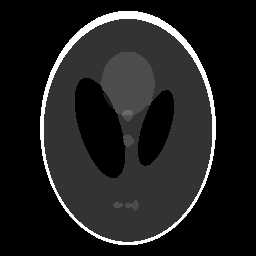}
	\caption{$\gamma = \protect\input{\figfolderA gamma.txt}$}
	\end{subfigure}
	\hfill
			\begin{subfigure}[t]{\figwidth}
	\includegraphics[width=\columnwidth]{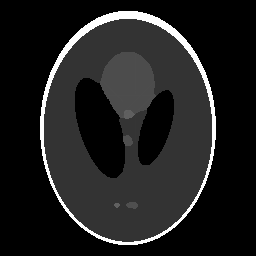}
	\caption{$\gamma = \protect\input{\figfolderB gamma.txt}$}
	\end{subfigure}
	\hfill
				\begin{subfigure}[t]{\figwidth}
	\includegraphics[width=\columnwidth]{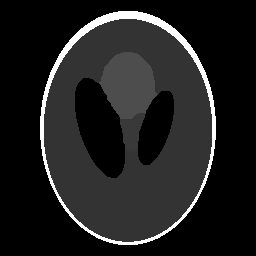}
	\caption{$\gamma = \protect\input{\figfolderC gamma.txt}$}
	\end{subfigure}
	\hfill
			\begin{subfigure}[t]{\figwidth}
	\includegraphics[width=\columnwidth]{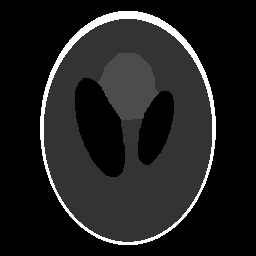}
	\caption{$\gamma = \protect\input{\figfolderD gamma.txt}$}
	\end{subfigure}
	\vspace{-0.2cm}
	\caption{Influence of the model parameter $\gamma$ on the result ($\protect\input{\figfolder nAngles.txt}$ noisefree projections of Radon data).
	  For higher  $\gamma,$ the small segments vanish but the gross structures are recovered.
}
	\label{fig:gamma}
\end{figure}

We demonstrate  
the robustness of our method to drastic angular undersampling (Figures \ref{fig:intro}, \ref{fig:radonAnglesNew},
\ref{fig:gamma}, and \ref{fig:noiseSparse}).
To set the results into context, 
we additionally show the results
of filtered backprojection (FBP).
FBP is the standard algorithm in many commercial CT scanners
\cite{pan2009commercial}.
We use the  Matlab function \texttt{iradon}, both with a standard Ram-Lak filter 
and with a Hamming window.
The  cutoff frequency for the Hamming window
was tuned in steps of $0.01$ with respect to PSNR. 
Recall that the PSNR is given by
\(
	\mathrm{PSNR}(u) = 10 \log_{10}(  ( m n \| g\|_\infty^2)/  \| g - u \|_2^2)
\)
where $g \in \R^{m\times n}$ is the ground truth.
For comparison, we further use the mean structural similarity index (MSSIM) \cite{wang2004image}.
We use Matlab's function \texttt{ssim} with standard parameters for the computation of the MSSIM. 
The MSSIM is bounded from above by $1.$ For both PSNR and MSSIM, higher values are better.  
We also compare with total variation regularization
which uses
the total variation $\alpha \| \nabla u \|_1,$  $\alpha >0,$ as regularizing term.
We follow the implementation of the Chambolle-Pock  algorithm \cite{chambolle2011first} provided by G.~Peyre \cite{peyre2011numerical}.

In Figure \ref{fig:radonAnglesNew}, 
we observe that the classical reconstruction methods perform poorly 
when using only few projection angles.
The standard FBP reconstruction produces streak artifacts which are typical for angular undersampling.
The FBP reconstruction using optimized Hamming window smoothes out the edges.
Total variation minimization achieves a high-quality reconstruction from $17$ projections,
but the quality decreases significantly for fewer angles.
(Compare similar observations in \cite{chartrand2007exact} for MRI-type data.)
In contrast, the proposed method achieves an almost perfect segmentation 
from as few as $7$ projections.

In Figure \ref{fig:intro}, we observe that the classical reconstruction combined 
with subsequent segmentation 
leads to poor results in a sparse angular setup.
For the segmentation part in Figure \ref{fig:intro}(c), we used the $\alpha$-expansion graph-cut algorithm based 
on max-flow/min-cut of the library \texttt{GCOptimization 3.0} of O.~Veksler  and A.~Delong \cite{boykov2001fast, boykov2004experimental, kolmogorov2004energy}, which is a state-of-the-art image segmentation algorithm.

In Figure \ref{fig:gamma}, we illustrate the  influence
of the model parameter $\gamma$ to the result. 
We observe that for higher $\gamma$ smaller details vanish
but the macrostructures are still recovered.

In Figure \ref{fig:noiseSparse}, we add Gaussian noise to the observations.
Due to the noise, the small structures vanish in the results.
Nevertheless, the large scale structures are  recovered reliably.

\begin{figure}[t]

\def\figwidth{0.23\columnwidth}
\def\figspace{\hspace{0.05\textwidth}}
\centering
\def\figfolder{newExperiments/sparseAngleNoise/}
	\def\figfolderA{newExperiments/sparseAngleNoise/noiseLevel0.005/}
		\def\figfolderB{newExperiments/sparseAngleNoise/noiseLevel0.01/}
					\def\figfolderC{newExperiments/sparseAngleNoise/noiseLevel0.02/}
	\begin{subfigure}[t]{\figwidth}
	\includegraphics[width=\columnwidth]{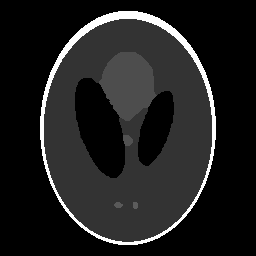}
	\caption{Noise level $\protect\input{\figfolderA sigma.txt},$\\
	($\gamma = \protect\input{\figfolderA gamma.txt},$ RI: $\protect\input{\figfolderA/randIndexPotts.txt}$).
	}
	\end{subfigure}
	\figspace
			\begin{subfigure}[t]{\figwidth}
	\includegraphics[width=\columnwidth]{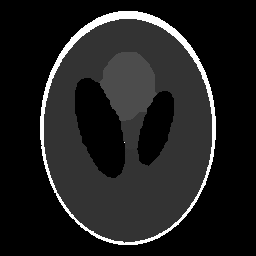}
	\caption{Noise level $\protect\input{\figfolderB sigma.txt},$\\
	($\gamma = \protect\input{\figfolderB gamma.txt},$ RI: $\protect\input{\figfolderB/randIndexPotts.txt}$).
	}
	\end{subfigure}
		\figspace
						\begin{subfigure}[t]{\figwidth}
	\includegraphics[width=\columnwidth]{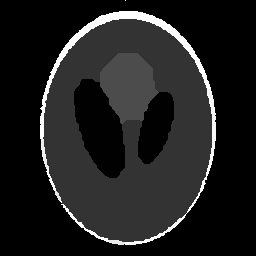}
	\caption{Noise level $\protect\input{\figfolderC sigma.txt},$ \\
	($\gamma = \protect\input{\figfolderC gamma.txt},$ RI: $\protect\input{\figfolderC/randIndexPotts.txt}$).
	}
	\end{subfigure}
	\vspace{-0.2cm}
	\caption{Effect of noise on the results ($\protect\input{\figfolder nAngles.txt}$ projection angles).
    The small structures vanish for higher noise levels whereas the large  structures are segmented reliably.
}
	\label{fig:noiseSparse}
\end{figure}

\subsection{Radon Data with Dense Angular Sampling}

% !TEX root =./inversePotts2D.tex
\begin{figure}
\def\figfolder{newExperiments/radonNoiseFull/NoiseLevel}
\def\figwidth{0.23\textwidth}
\def\figspace{\hspace{0.05\textwidth}}
\centering
			\def\noiseLevelA{0.05}
			\def\noiseLevelB{0.1}
	\def\noiseLevelC{0.15}
	
	\centering
	\begin{subfigure}[t]{\figwidth}
		\centering
		 \includegraphics[width=\columnwidth]{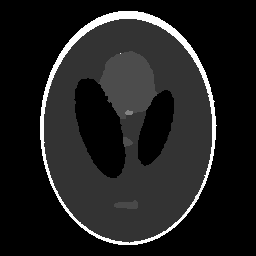} 
	\caption{Noise level \noiseLevelA,
		($\gamma = \protect\input{\figfolder\noiseLevelA/gamma.txt},$ RI: $\protect\input{\figfolder\noiseLevelA/randIndexPotts.txt}$).
	}
	\end{subfigure}
	\figspace
		\begin{subfigure}[t]{\figwidth}
		\centering
		 \includegraphics[width=\columnwidth]{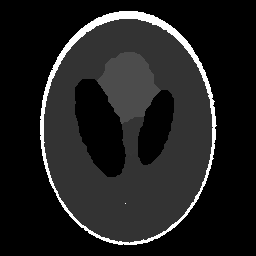} 
	\caption{Noise level \noiseLevelB, 
		($\gamma = \protect\input{\figfolder\noiseLevelB/gamma.txt},$ RI: $\protect\input{\figfolder\noiseLevelB/randIndexPotts.txt}$).
	}
	\end{subfigure}
	\figspace
		\begin{subfigure}[t]{\figwidth}
		\centering
		 \includegraphics[width=\columnwidth]{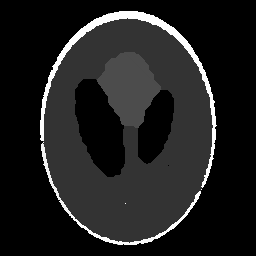} 
	\caption{Noise level \noiseLevelC,
		($\gamma = \protect\input{\figfolder\noiseLevelC/gamma.txt},$ RI: $\protect\input{\figfolder\noiseLevelC/randIndexPotts.txt}$).
		}
	\end{subfigure}
	\caption{%Reconstruction of the Shepp-Logan phantom 
   In the case of densely sampled Radon data (360 projection angles), 
   our method is able to segment the large geometric  structures for very high noise levels. 
   %(See Figure \ref{fig:intro}a for the original image.)
	}
\label{fig:radonNoiseFull}
\end{figure}

The costs of evaluating $\Rc$ and $\Rc^*$ increases
with the number of available projection angles.
In total we need between  $10^5$ and $10^6$ evaluations
of $\Rc$ and $\Rc^*.$
Therefore, using the conjugate gradients methods to solve 
 the Tikhonov problem \eqref{eq:generalADMM} can be time consuming
 for a dense angular sampling.
In this setup, we can use the following efficient
alternative implementation.
Here, the minimizer of the Tikhonov problem is computed using a 
filtered-backprojection-type formula
with a special filtering function which we describe next.
Recall that
the backprojection operator $\Rc^\ast$ is defined via
\begin{equation*}
	\Rc^\ast f (x) = \int_{S^1} f(\theta,x\cdot \theta)\d\theta.
\end{equation*}
Let $f = \Rc w$ for some $w \in L^2(\Omega)$ and $\alpha >0 $. Then,
the solution $v^*$ of the Tikhonov problem is given by
\begin{equation}
\label{eq:fbpTikhonov}
 	v^* = \argmin_{v \in L^2(\Omega)} \norm{\Rc v - f}_2^2 + \alpha\norm{v - z}_2^2
=\Rc^\ast H_{\alpha} (f-\Rc z),
\end{equation}
where the filtering operator is defined via
\[
  H_\alpha f  =  \Fc^{-1}_s ( h_\alpha \Fc_s f)
\]
with the filter function
\[
	h_\alpha(r) = \frac{\abs{r}}{4\pi+\alpha\abs{r}}.
\]
Above, $\Fc_s$ (and $\Fc^{-1}_s$) denotes the one dimensional Fourier transform (and its inverse) of a function
$f(\theta,s)$ with respect to the parameter $s.$
Since we have not found the statement in this form in the literature, 
we provide a short proof in the Appendix.
We remark that this procedure can be applied only to densely sampled data.
In a sparse angle setup, 
it produces unacceptably large errors.

In Figure \ref{fig:radonNoiseFull}, we  show  the result of our method
for the reconstruction of the Shepp-Logan phantom from Radon data with dense angular sampling (360 angles). 
For  very high noise levels the small details vanish but the large geometric structures are 
still segmented reliably.

\subsection{Real Radon Data from a PET Device}
Next, we apply our method to PET data. 
The underlying PET model generates ideal data of the form $a \cdot \Rc g$
where $\Rc g$ is the Radon transform of the imaged object $g$ and 
$a$ is a known function depending on the attenuation. 
Eliminating the known function $a,$ 
 we are exactly in the setup of the classical Radon transform. 
In Figure \ref{fig:PET}, we see the results for PET data
of a physical thorax phantom \cite{fessler1997grouped}.
Our method is able to jointly reconstruct and segment the anatomic structures (lung, spine, and thorax body) 
 even from sparsely sampled data.

\begin{figure}[t]
\def\figfolder{experiments/pet/}
\def\figwidth{0.23\columnwidth}
\centering
	\begin{subfigure}[t]{\figwidth}
		\centering
	\includegraphics[width=1\columnwidth]{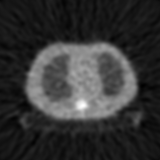}
	\caption{FBP using 192 projections.}
	\end{subfigure}
	\hfill
	\begin{subfigure}[t]{\figwidth}
	\centering
	\includegraphics[width=1\columnwidth]{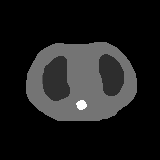}
	\caption{Proposed method using 192 projections.}
	\end{subfigure}
	\hfill
	\begin{subfigure}[t]{\figwidth}
	\centering
	\includegraphics[width=1\columnwidth]{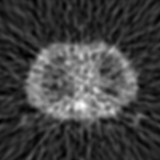}
		\caption{FBP using 24 projections.}
	\end{subfigure}
	\hfill
		\begin{subfigure}[t]{\figwidth}
	\centering
	\includegraphics[width=1\columnwidth]{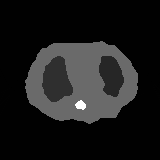}
	\caption{Proposed method using 24 projections.}
	\end{subfigure}
	\caption{
 Joint reconstruction and segmentation
  from undersampled PET data of a physical thorax phantom \cite{fessler1997grouped}. 
 The proposed method segments the anatomic structures 
 (lung, spine, and thorax body).
 Filtered backprojection results are shown 
 for comparison.
}
	\label{fig:PET}
\end{figure}

\section{Further Applications}

We briefly discuss further applications. We consider spherical Radon data  as well as blurred data.

\subsection{Spherical Radon Data}

The reconstruction of a function $u:\R^2\to\R$ from spherical averages plays an important role in the context of TAT/PAT. It has been intensively studied in recent years and still attracts much attention, cf. \cite{Ansorg:2012fk,Buehler:2011bv,Finch:2004hs,Haltmeier:2014hn,Kunyansky:2007dw,Natterer:2012hh,Wang:2009uz, Elbau:2012df} to mention only a few.  In the next example, we apply our method to this problem. Here, we assume the data are given by the spherical mean Radon transform
%The spherical mean Radon transform $\Mc$ is given by 
\begin{equation}
\label{eq:spherical radon transform}
	\Mc u(\theta(\varphi),t)=
	\displaystyle\int_{S^1} u(\theta(\varphi) + t\mathbf \zeta)\, \d\zeta,
\end{equation}
for some angles $\varphi \in [0,2\pi]$ and some radii $t\in(0,2]$, where $\theta(\varphi) = (\cos\varphi,\sin\varphi)$. In our experiment, we computed the spherical means of the Shepp-Logan head phantom for 7 equispaced angles and $512$ equispaced radii. (See Figure \ref{fig:intro}a for the original image.) 
The Tikhonov subproblem is solved using the standard conjugate gradient algorithm on the normal equation. For comparison, we show the result of filtered-backprojection-type reconstruction algorithm as proposed in \cite{Ansorg:2012fk,Seyfried2014} using R.~Seyfried's implementation of the algorithm. As can be observed in Figure~\ref{fig:sphericalRadon}, the experimental results are similar to those of the classical Radon transform in Section \ref{sec:applications}. The FBP-type reconstruction and the total variation reconstruction suffer from severe artifacts when only few data are available.
In contrast, our method almost perfectly recovers the original image.

    \begin{figure}[t]

\def\figwidth{0.23\columnwidth}
\def\figspace{\hspace{0.05\textwidth}}
\centering
	\def\figfolder{newExperiments/sphericalRadonNew/}
	\begin{subfigure}[t]{\figwidth}
	\includegraphics[width=\columnwidth]{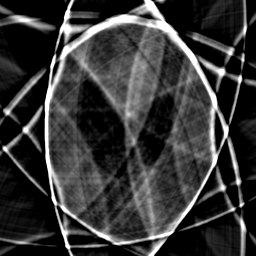}
	\caption{FBP-type \\
	 	 (PSNR: 
$\protect\input{\figfolder psnrFBP.txt},$ MSSIM: ${\protect\input{\figfolder ssimFBP.txt}}$)}
	\end{subfigure}
	\figspace
		\begin{subfigure}[t]{\figwidth}
	\includegraphics[width=\columnwidth]{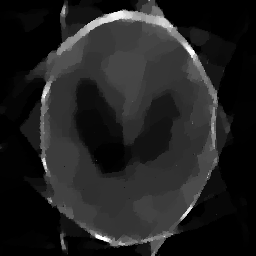}
	\caption{Total variation, $\alpha = \protect\input{\figfolder lambda.txt}$
		 (PSNR: 
$\protect\input{\figfolder psnrTV.txt},$ MSSIM: ${\protect\input{\figfolder ssimTV.txt}}$)	}
\end{subfigure}
	\figspace
		\begin{subfigure}[t]{\figwidth}
	\includegraphics[width=\columnwidth]{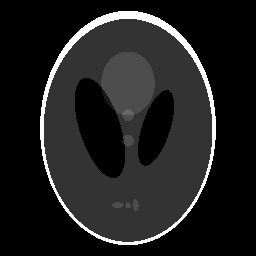}
	\caption{Our method, $\gamma = \protect\input{\figfolder gamma.txt}$
	 (PSNR: 
$\protect\input{\figfolder psnrPotts.txt},$ MSSIM: ${\protect\input{\figfolder ssimPotts.txt}},$  RI: ${\protect\input{\figfolder/randIndexPotts.txt}}$)
}
	\end{subfigure}
	\vspace{-0.2cm}
	\caption{
The proposed method reliably recovers all segments of the Shepp-Logan phantom 
from $\protect\input{\figfolder nAngles.txt}$ noisefree projections of spherical Radon data. 
}
	\label{fig:sphericalRadon}
\end{figure}

\subsection{Blurred Data}

We finally demonstrate the applicability of our method
to deblurring problems.
Here, the operator $A$ is a convolution operator.
Hence, the normal equation \eqref{eq:normalEq} 
can be solved efficiently by fast Fourier transform techniques.
In this experiment we particularly illustrate that our method can be applied to vector-valued data such as color images. 
We follow the splitting strategy as proposed in the present paper
and extend the univariate Potts problem to vector-valued data according to \cite{storath2014fast}.
It is worth mentioning that the computational effort grows only linearly in the
dimension of the vectorial data. 
For example, the cost for processing a color image is about three times the cost of processing a gray-value image.
In Figure \ref{fig:gaussianblur}, we show the joint reconstruction and segmentation  of an image blurred by a Gaussian kernel. 
(The image was taken from the Berkeley Segmentation Dataset \cite{MartinFTM01}). 
In Figure \ref{fig:motionblur}, we see 
the restoration of a traffic sign from simulated motion blur.
Motion blur is modeled as a one-dimensional convolution along the direction $v \in \R^2.$ 
Here, we use a moving average with respect to the horizontal direction.
This experiment also illustrates 
that a positive coupling-parameter sequence $\nu_k $ can improve the result.

\begin{figure}
\def\figfolder{experiments/deconvolution/gaussian/}
\def\figwidth{0.31\columnwidth}
\def\figspace{\hspace{0.02\columnwidth}}
\centering
	\begin{subfigure}[t]{\figwidth}
	\includegraphics[width=\columnwidth]{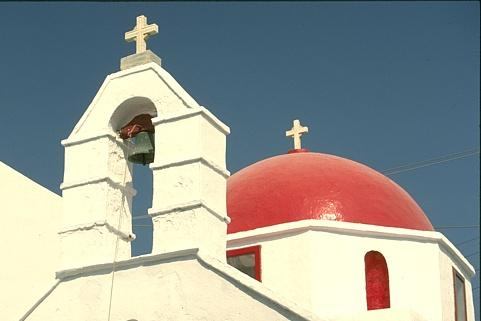}
	\caption{Original (${\input{\figfolder dim.txt}}$).}
	\end{subfigure}
	\figspace
	\begin{subfigure}[t]{\figwidth}
	\includegraphics[width=\columnwidth]{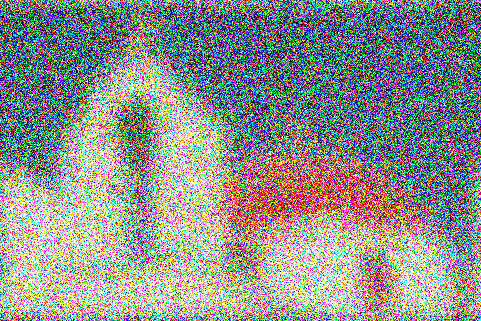}
	\caption{Blurred and noisy data.}
	\end{subfigure}
	\\[1ex]
	\begin{subfigure}[t]{\figwidth}
	\includegraphics[width=\columnwidth]{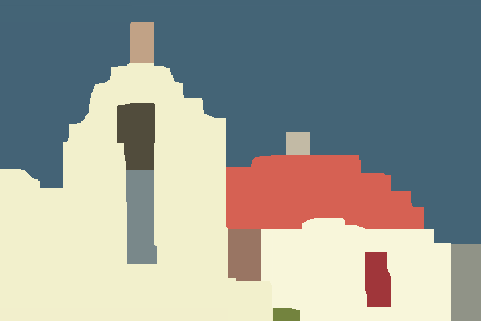}
	\caption{Result using $\Nc_0,$ }
	\end{subfigure}
	\figspace	
	\begin{subfigure}[t]{\figwidth}
	\includegraphics[width=\columnwidth]{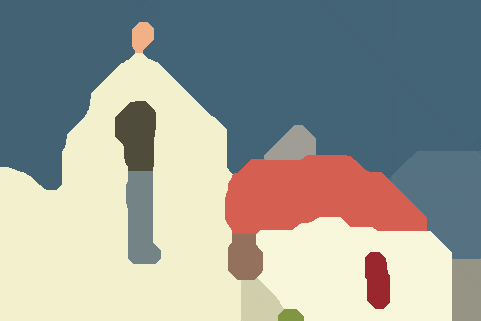}
	\caption{Result using $\Nc_1.$}
	\end{subfigure}
	\figspace
	\begin{subfigure}[t]{\figwidth}
	\includegraphics[width=\columnwidth]{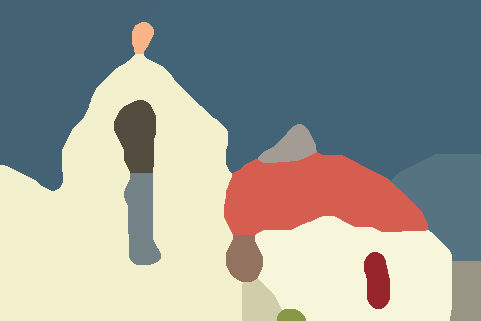}
	\caption{Result using $\Nc_2.$}
	\end{subfigure}
	\caption{
Joint reconstruction and segmentation of an image blurred by  a Gaussian kernel of standard deviation $\protect\input{\figfolder stdDevKernel.txt}$ and corrupted by extreme Gaussian noise of level $\protect\input{\figfolder sigma.txt}.$ We used $\gamma = {\protect\input{\figfolder gammaAniso.txt}}$. The higher the degree of isotropy
becomes, the smoother are the segment boundaries.}
	\label{fig:gaussianblur}
\end{figure}

\begin{figure}
\def\figfolder{experiments/deconvolution/stop/}
\def\figwidth{0.22\columnwidth}
\def\figspace{\hfill}
\centering
	\begin{subfigure}[t]{\figwidth}
	\includegraphics[width=\columnwidth]{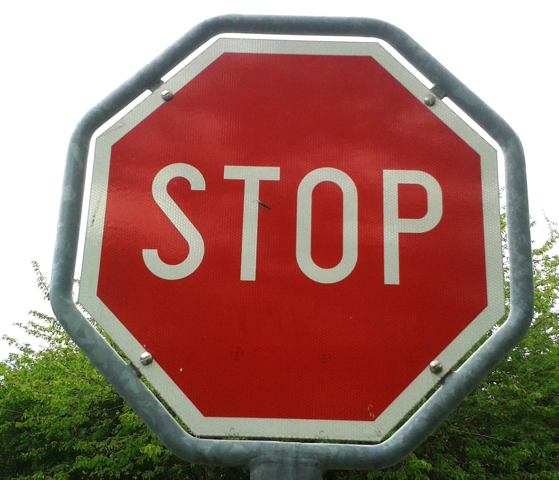}
	\caption{Original (${\input{\figfolder dim.txt}}$).}
	\end{subfigure}
	\figspace
	\begin{subfigure}[t]{\figwidth}
	\includegraphics[width=\columnwidth]{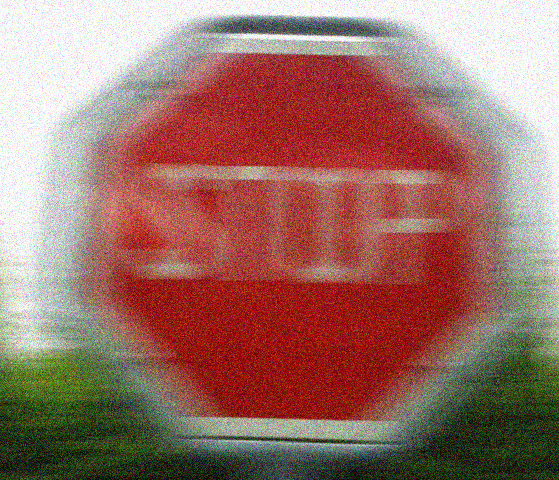}
	\caption{Blurred and noisy data.}
	\end{subfigure}
	\figspace
	\begin{subfigure}[t]{\figwidth}
	\includegraphics[width=\columnwidth]{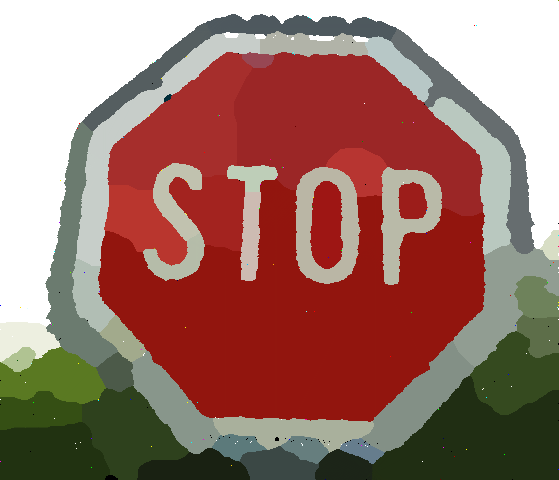}
	\caption{Result using $\nu_k = 0.$
	(MSSIM: $\protect\input{\figfolder ssimPotts.txt})$
	}
	\end{subfigure}
	\figspace
		\begin{subfigure}[t]{\figwidth}
	\includegraphics[width=\columnwidth]{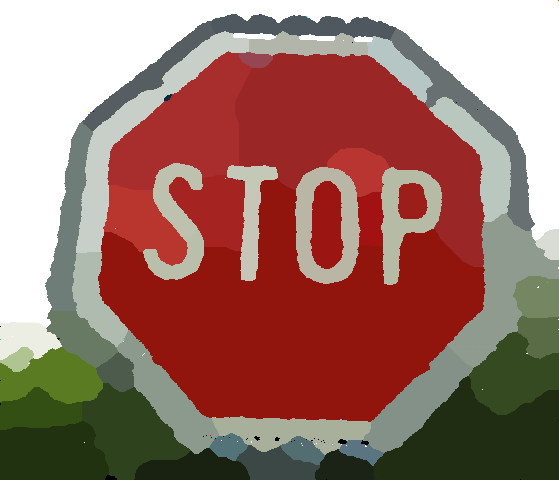}
	\caption{Result using $\nu_k =\mu_k/S.$ (MSSIM: $\protect\input{\figfolder ssimPottsRho.txt})$}
	\end{subfigure}
	\caption{
Segmentation of an image from a simulated horizontal motion blur of $\protect\input{\figfolder length.txt}$ pixel length with Gaussian noise of level $\protect\input{\figfolder sigma.txt}.$ 
The letters are recovered almost perfectly (c).
The result improves if we use the algorithm of \eqref{eq:generalADMM} with $\nu_k = \mu_k/S$ instead of $\nu_k = 0$ (d).
}
	\label{fig:motionblur}
\end{figure}

\section{Conclusion}
In this paper, 
we have developed a new splitting approach for the Potts model 
(or piecewise-constant Mumford-Shah model) for ill-posed inverse problems in imaging. 
We have presented a general discretization scheme 
which permits near-isotropic approximations of the length terms.
This discretization allowed us to split the Potts problem 
into specific subproblems that can be solved with efficient algorithms.
We have demonstrated 
the capability of our method in various imaging applications.
In particular,  our algorithm has reconstructed all segments of the Shepp-Logan phantom from only seven projections of
Radon and spherical Radon data, respectively.
Further, we  have obtained high-quality 
results from highly incomplete and noisy data.
Finally, we have demonstrated that it is applicable for joint reconstruction and segmentation 
of real tomographic data.

\appendix

\section{Proof of Equation \eqref{eq:fbpTikhonov}}
\begin{proof}[Proof of Equation \eqref{eq:fbpTikhonov}]
	Setting $u = v-z$ and $g = f - \Rc z $
 we first rewrite the cost functional as
	\begin{equation}
	\label{auxeq:fbpTikhonov}
	\norm{\Rc v - f}_2^2 + \alpha\norm{v - z}_2^2 = 
		\norm{\Rc u - g}_2^2 + \alpha\norm{u}_2^2.
	\end{equation}
In the following, we use the notation  $u_\alpha$ to denote the minimizer of the right-hand side in \eqref{auxeq:fbpTikhonov}. We first note that $u_\alpha$ satisfies the normal equation $\Rc^\ast\Rc u_\alpha + \alpha u_\alpha = \Rc^\ast g$ (see \eqref{eq:normalEq}). Further, we get from \cite[Theorem II.1.5]{Natterer86} that $\Rc^\ast\Rc u_\alpha =2(\norm{\,\bdot\,}_2^{-1}\conv \ u_\alpha).$ Therefore, by taking the Fourier transform on both sides of the normal equation and applying the convolution theorem together with \cite[Ch.V, Lemma 5.2]{Helgason99}, we obtain the following relation for the Fourier transform of $u_\alpha:$
\begin{equation}
\label{eq:Fourier Tikhnov minimizer}
	\hat u_\alpha(\xi) = \frac{\norm{\xi}_2}{4\pi+\alpha\norm{\xi}_2}\widehat{\Rc^\ast g}(\xi).
\end{equation}
Formula \eqref{eq:Fourier Tikhnov minimizer} is a special case of a formula derived in \cite[Sec. 7]{Jonas:2001kv}. Now, since $g=\Rc w'$ with $w'=w-z$, a similar argument shows that $\widehat{\Rc^\ast \Rc w'}(\xi) = 4\pi\norm{\xi}_2^{-1} \widehat{w'} (\xi)$ and, therefore, that
\begin{equation}
\label{eq:Fourier Tikhnov minimizer2}
	\hat u_\alpha(\xi) = \frac{4\pi}{4\pi+\alpha\norm{\xi}_2} \widehat{w'}(\xi).
\end{equation}

Next,  
we use  the Fourier-slice theorem \cite[Theorem II.1.1]{Natterer86}. It states that
$\widehat{\Rc_{\theta}w'}(r) = \sqrt{2\pi}\ \widehat{w'}(r\theta)$ , where $\Rc_\theta w' = \Rc w'(\theta,\cdot\ )$. Using \eqref{eq:Fourier Tikhnov minimizer2} we obtain
\begin{align*}
	u_\alpha(x) &= \frac{1}{2\pi} \int_{\R^2}  \widehat{u}_\alpha(\xi) e^{ix\cdot\xi}\d\xi\\
	&= \frac{1}{2\pi} \int_{\R^2} \frac{4\pi}{4\pi+\alpha\norm{\xi}}_2  \widehat{w'}(\xi) e^{ix\cdot\xi}\d\xi\\
	&= \int_{S^{1}} 2\int_{0}^\infty  \frac{1}{4\pi+\alpha\abs{r}} \widehat{w'}
	(r\theta)e^{irx\cdot\theta}r\d r\d\theta\\
	&= \int_{S^{1}} \int_{-\infty}^\infty  \frac{\abs{r}}{4\pi+\alpha\abs{r}} \widehat{w'}(r\theta)e^{irx\cdot\theta}\d r\d\theta\\
	&= \int_{S^{1}} \frac{1}{\sqrt{2\pi}}\int_{-\infty}^\infty  \frac{\abs{r}}{4\pi+\alpha\abs{r}}\widehat{\Rc_{\theta}w'}(r)e^{irx\cdot\theta}\d r\d\theta\\
	&=  \int_{S^{1}}H_\alpha \Rc w'(\theta,x\cdot\theta) \d\theta\\
	&=  \Rc^\ast H_\alpha \Rc w'(x).
\end{align*}
Since $\Rc w'=\Rc w - \Rc z = f - \Rc z$, the assertion follows.
\end{proof}

\section*{Acknowledgements}
The research leading to these results has received funding from the European Research Council under the European Union's Seventh Framework Programme (FP7/2007-2013) / ERC grant agreement no.~$267439.$
The second and the third author acknowledge support by the Helmholtz Association within the young investigator group VH-NG-526. 
We would like to thank 
Eric Todd Quinto for valuable discussions on the topic of incomplete data tomography,
Philippe Th\'evenaz for a valuable discussion on medical imaging,
 Ruben Seyfried for providing us the implementation of his summability method for PAT, 
and Jeffrey Fessler for making his PET dataset publicly available.
Moreover, we would like to thank the anonymous reviewers
for their valuable comments and suggestions which 
helped to improve the paper.

\section*{References}

\bibliographystyle{abbrv}
\bibliography{references}

\end{document}